\documentclass[12pt,a4paper,english,reqno]{amsart}
\usepackage[T1]{fontenc}
\usepackage[utf8]{inputenc}
\usepackage{lmodern}
\usepackage{microtype}
\usepackage{xcolor}
\newif\ifshowchanges
\showchangestrue
\newcommand{\rev}[1]{{\ifshowchanges\color{blue}\fi #1}}

\usepackage{etoolbox}
\patchcmd{\abstract}{\leftmargin3pc \rightmargin\leftmargin}
  {\leftmargin0pt \rightmargin0pt}{}
  {\PackageError{dyck-revision}{Abstract width patch failed}{Check amsart.cls.}}
\usepackage{mathrsfs}
\usepackage{amssymb}
\usepackage{graphicx}
\usepackage{calrsfs}
\DeclareMathSizes{11}{11}{7}{5}
\usepackage{tikz}
\usetikzlibrary{arrows.meta}

\usepackage{enumitem}
\usepackage[colorlinks=true,allcolors=blue]{hyperref}
\usepackage[nameinlink,noabbrev]{cleveref}

\theoremstyle{plain}
\newtheorem{theorem}{Theorem}[section]
\newtheorem{proposition}[theorem]{Proposition}
\newtheorem{lemma}[theorem]{Lemma}
\newtheorem{corollary}[theorem]{Corollary}
\newtheorem*{theorem*}{\rev{Theorem}}
\newtheorem*{proposition*}{\rev{Proposition}}
\newtheorem*{lemma*}{\rev{Lemma}}
\newtheorem*{corollary*}{\rev{Corollary}}

\theoremstyle{definition}
\newtheorem{definition}[theorem]{Definition}
\newtheorem{exm}[theorem]{Example}
\newtheorem{remark}[theorem]{Remark}
\newtheorem{Prob}[theorem]{Problem}
\newtheorem*{definition*}{\rev{Definition}}
\newtheorem*{exm*}{\rev{Example}}
\newtheorem*{remark*}{\rev{Remark}}
\newtheorem*{Prob*}{\rev{Problem}}

\newcommand{\kk}{\Bbbk}
\newcommand{\perf}{\operatorname{per}}
\newcommand{\Hom}{\operatorname{Hom}}
\newcommand{\End}{\operatorname{End}}
\newcommand{\RHom}{\operatorname{\mathbf{R}Hom}}
\newcommand{\thick}{\operatorname{thick}}
\newcommand{\Dyck}{\operatorname{Dyck}}
\newcommand{\gldim}{\operatorname{gldim}}
\newcommand{\add}{\operatorname{add}}
\newcommand{\op}{\mathrm{op}}

\newcommand{\C}{\mathcal C}
\newcommand{\X}{\mathcal X}
\newcommand{\G}{\Gamma}
\newcommand{\Gs}{\Gamma^{\sharp}}

\newcommand{\os}{\operatorname{os}}
\newcommand{\Wcat}{\mathcal W}
\newcommand{\Fcat}{\mathcal F}
\newcommand{\modcat}{\ensuremath{\mbox{\rm -mod}}}
\newcommand{\Ext}{\operatorname{Ext}}
\newcommand{\Db}[1]{{\mathscr D}^b(#1)}
\newcommand{\ovr}{\rm \overrightarrow}
\newcommand{\lra}{\longrightarrow}

\newcommand{\Ob}{\rm Ob}

\newcommand{\pmodcat}[1]{#1\mbox{{\rm -proj}}}
\newcommand{\Kb}[1]{{\mathscr K}^b(#1)}

\setlist{itemsep=2pt,topsep=4pt}
\allowdisplaybreaks
\title[Relative tilting, interlacing, and Dyck posets]
{Relative tilting, interlacing categories and rational Dyck posets}
\author{Shengyong Pan}
\date{\today}
\hypersetup{
  pdfauthor={Shengyong Pan},
  pdftitle={Relative Tilting, Interlacing Categories, and Rational Dyck Posets},
  pdfsubject={Derived equivalences, staircase posets, rational Dyck paths, and higher Auslander algebras},
  pdfkeywords={interval tilting, higher Auslander algebra, rational Dyck path, derived equivalence, Fukaya category}
}
\begin{document}

\begin{abstract}
\color{black}
\showchangesfalse
We construct relative interval tilts for finite linear categories,
with no restriction on global dimension. Applied to a skew Ferrers
diagram, the construction identifies the derived category of its
ideal-lattice incidence algebra with that of its horizontal-strip
algebra, realized as a higher Auslander corner. A separate weighted
exchange theorem recovers all boundary labels from cyclic Dyck seeds
for arbitrary positive parameters; the interlacing separation required
for rigidity holds exactly when the parameters are coprime.
Exact deletion sequences realize these exchanges in the derived
category, yielding the Dyck Serre-window tilt. \rev{Inverse intrinsic
Serre functors untwist its layers, and unit extensions directly
produce a chain-product tilting generator.} The Chapoton--Ladkani--Rognerud conjecture follows:
for coprime $a,b$, the incidence algebras of
$\ovr A_{a+b}\times\Dyck_{a,b}$ and the full path lattice in the
$a\times b$ rectangle are derived equivalent. We also recover the
window's global-dimension and higher cluster-tilting properties from
a complete-window homology calculation and standard Auslander--Reiten
theory. The diagram categories admit models as admissible subcategories
of partially wrapped Fukaya and Fukaya--Seidel categories, generated
by selected product Lagrangians and thimbles. The relative tilts have
cubical twisted-complex presentations, and product-arc triangles give
embedded representatives for the penultimate collection in every
symmetric-product degree at least two. Finally, the Serre functor $S$
of $\Db{\kk\Dyck_{a,b}}$ satisfies
$S^{a+b+1}\cong[(a-1)(b-1)]$ for coprime $a,b$, while a rank obstruction
excludes the unchanged CLR formula for non-coprime parameters.
\end{abstract}

\subjclass[2020]{16G20, 18G80, 05E16, 53D37}
\keywords{derived equivalence, interval tilting, staircase poset, higher Auslander algebra, rational Dyck path, partially wrapped Fukaya category}
\address{School of Mathematics and Statistics, Beijing Jiaotong University,
Beijing, 100044, China.}
\address{Beijing Key Laboratory of Biological Big Data and Topological Statistics, Beijing Jiaotong University,
Beijing, 100044, China.}
\email{shypan@bjtu.edu.cn}
\maketitle

\section{Introduction}

Two constructions organize the paper. Relative Kan-extension tilting
inserts one interlacing inequality at a time, transforming diagram
incidence categories into horizontal-strip categories. Independently,
weighted deletion exchanges recover boundary labels from cyclic Dyck
seeds. Exact sequences then turn this combinatorial recovery into
generation by interval modules. The exchange construction works for
every slope; coprimality is exactly the condition for rigidity.
Together, these constructions give the diagram and Dyck equivalences,
and the Chapoton--Ladkani--Rognerud conjecture becomes a corollary.
The subsequent homological and symplectic results explain the structure
of these equivalences. Throughout, Dyck paths carry the coordinatewise
order on horizontal-step positions. Section~\ref{sec:dyck} specifies
the above- and below-diagonal conventions, and
$\ovr A_r=\{1<\cdots<r\}$ denotes the linearly ordered set.

The relative construction extends the poset interval tilt of
\cite[Theorem~3.1]{CLR} to finite linear categories. It allows
arbitrary finite-dimensional morphism spaces, non-semisimple
endomorphism algebras at individual objects, and zero composites.

\begin{theorem}[Theorem~\ref{thm:relative-without-gldim}]

Let $\X$ be a finite $\kk$-linear category, $Y$ a finite poset,
and $F(y)\subseteq\X$ an increasing family of full sieves.
For the categories $\G(\X,Y,F)$ and $\Gs(\X,Y,F)$ of
Section~\ref{sec:relative}, the exact right Kan extensions
\[
 T=\bigoplus_{y\in Y}\ \bigoplus_{x\in F(y)}
 (\iota_y)_*\Hom_{F(y)}(x,-)
\]
form a tilting complex, with
$\End(T)^{\op}\cong\operatorname{Alg}(\Gs)$.
Thus $\Db{\G}\simeq\Db{\Gs}$, without a finite-global-dimension
hypothesis. If $Y\ne\varnothing$, then
$\operatorname{pd}_{\G}T\leq |Y|-1$.
\end{theorem}

The same layer filtration proves both perfectness and generation.
If $\X$ is homologically directed, all fibers and total
categories have finite global dimension
(Corollary~\ref{cor:directed-relative}).
The target is the boundary-ideal quotient
$\Gs=\G/\mathcal I$ of Lemma~\ref{lem:associativity}.
The filtration constructs a tilting generator for this quotient.
In coordinates, each application inserts a single inequality
$y_j<x_{j+1}$. Repeating it turns incidence multiplication into
the interlacing multiplication of a higher Auslander corner.

This has an intrinsic formulation in terms of skew Ferrers
diagrams. For partitions $\mu\subseteq\lambda$, give
$\Theta=\lambda\setminus\mu$ the induced cell order and write
$J(\Theta)$ for its ideal lattice. A horizontal strip has at
most one cell in each column. The category
$\mathcal H_{\lambda/\mu}$ has intermediate partitions
$\mu\subseteq\nu\subseteq\lambda$ as objects, one basis morphism
$\nu\to\eta$ when $\eta/\nu$ is a horizontal strip, and nonzero
composition precisely when the total increment is again a
horizontal strip.

\begin{theorem}[Theorems~\ref{thm:coordinate-intervals}
and~\ref{thm:skew-strip-tilting}]
\label{thm:intro-staircase}

If $\mu\subseteq\lambda$ lie in an $m\times n$ rectangle, then
\[
 \Db{\kk J(\lambda\setminus\mu)}
 \simeq\Db{\mathcal H_{\lambda/\mu}}
 \simeq\Db{e_{\lambda/\mu}A_{n+1}^{m}e_{\lambda/\mu}}.
\]
The equivalence is a composite of $m-1$ explicit relative tilts.
Transposition gives
$\Db{\mathcal H_{\lambda/\mu}}\simeq
\Db{\mathcal H_{\lambda^T/\mu^T}}$.
\end{theorem}

The case $\mu=\varnothing$, $\lambda=(n^m)$ is Gottesman's
rectangular equivalence. Coordinate staircases correspond to
ordinary Ferrers diagrams; arbitrary coordinate intervals
correspond to skew diagrams. Proposition~\ref{prop:noninterval-counterexample}
shows why arbitrary coordinate lower ideals cannot replace
intervals in this statement.

Dyck corners now become a particular diagram rather than a
separate algebra construction. For $a,b\geq2$, put
 $s=\left\lceil\frac ab\right\rceil,
 t=\left\lceil\frac ba\right\rceil,
 m=a-s, n=b-t,
 \lambda_{a,b}
 =\left(\left\lfloor\frac{b(a-1)}a\right\rfloor,\ldots,
         \left\lfloor\frac{bs}a\right\rfloor\right)$.
Deleting the forced prefix identifies below-diagonal paths with
$J(\lambda_{a,b})$. Their horizontal-strip algebra is
$B^{\mathrm{st}}_{a,b}$, with corner realization in $A_{n+1}^{m}$.
For coprime parameters it is the Dyck corner $B_0$
(Theorem~\ref{thm:intrinsic-dyck-corner}). Consequently,
 $$
 \Db{\kk\Dyck^{\mathrm{below}}_{a,b}}\simeq\Db{B^{\mathrm{st}}_{a,b}},
 \Db{\kk\Dyck^{\mathrm{above}}_{a,b}}\simeq\Db{(B^{\mathrm{st}}_{a,b})^{\op}}.
 $$
The opposite in the above-diagonal equivalence comes from a half-turn of the
rectangle. Transposition gives
$\lambda_{b,a}=\lambda_{a,b}^T$ and hence
$\Db{B^{\mathrm{st}}_{a,b}}\simeq\Db{B^{\mathrm{st}}_{b,a}}$
for all parameters. If $a=1$ or $b=1$, the Dyck poset is a
singleton and is treated directly.

The second structural theorem is a statement about finite ordered
labels. A deletion exchange recovers one $(d+1)$-element face of
a $(d+2)$-subset from its other faces. The positive height and
defect of a label are defined by the weighted cut in
\eqref{eq:window-integer-measure}.

\begin{theorem}[Theorem~\ref{thm:weighted-exchange-completion}]

For every $d,n\geq1$, rotations of the labels $E\ell$, with
$\ell\in\Dyck^{\mathrm{below}}_{d,n}$, recover every label of
positive height by deletion exchanges. They recover, in particular,
every boundary label ending in $d+n+1$. The recovery is well-founded:
each dependency strictly decreases $(n-h,\mu)$.
There are no interlacing pairs between $E\ell$ and
$\operatorname{rot}^q(E\ell')$ for $2\leq q\leq d+n$ if and only
if $\gcd(d,n)=1$.
\end{theorem}

By the label theorem, we obtain $(d+n)|\Dyck^{\mathrm{below}}_{d,n}|\geq\binom{d+n}{d}$,
with equality exactly for coprime parameters. The equality is the
classical rational Catalan count; see \cite[Section~3.3]{CLR}.
The strict inequality provides the non-coprime rank obstruction.

Exact deletion sequences lift this combinatorial closure to thick
closure of interval modules. The Serre-window assertion is now
a consequence, Corollary~\ref{cor:coordinate-window}: the window
generates for every slope, and is tilting exactly for coprime
parameters. This is where the local boundary exchange and the
interlacing separation enter; they are proved independently of
the derived equivalence they produce.

The passage from the window to the tensor product is direct.
Theorem~\ref{thm:direct-Serre-window} untwists each layer by the
inverse intrinsic Serre functor of the seed algebra. The resulting
objects have the Ext pattern of the simples of a chain tensored
with that algebra. A degree-and-endpoint calculation proves
formality, and successive unit extensions give the projective
collection. This construction specifies the tensor-product
tilting object. The known replicated-algebra route gives a shorter
proof after invoking Ladkani's equivalence; the two routes have
the same scope for CLR. We use the direct construction below.
Thus
Corollary~\ref{cor:direct-window-tensor} gives
$\Db{A_{n+1}^{d}}\simeq
\Db{\kk\ovr A_{d+n}\otimes B^{\mathrm{st}}_{d,n}}$
for coprime $d,n\geq2$, with an explicit tilting generator.

The original window has a different endomorphism algebra
$\Gamma=\End(T)$. Its known homological properties are
\[
 \gldim\Gamma=dn,\qquad S_\Gamma^{d+n+1}\cong[dn].
\]
Moreover, $\add(\Gamma\oplus D\Gamma)$ is
$dn\mathbb Z$-cluster tilting.
They follow separately from Corollary~\ref{cor:window-homology},
using the standard inputs of Section~\ref{subsec:cluster-preliminaries}.
The tensor-product generator and the original window need not
have the same endomorphism algebra.

The objects used here are Xing's Dyck Serre window and corner
\cite[Section~4.4 and Proposition~4.33]{Xing}.
The corner is the horizontal-strip algebra of the Dyck diagram;
the window places its rotated projective seeds in the full higher
Auslander category. Their diagram and direct tensor equivalences are
proved below from relative tilting, weighted exchanges, and Serre untwisting. The
coordinate argument develops Xing's height-and-defect method
\cite[Lemma~4.23]{Xing}. We also include a proof of the known
complete-window homology statement \cite[Theorem~3.6]{Xing},
using the standard Auslander--Reiten criterion and natural ambient
Serre period stated in Section~\ref{subsec:cluster-preliminaries}.
The recovered equivalences and window properties retain their original
attribution.

For coprime positive integers $a,b$, let $L_{a,b}$ be the lattice
of all north-east paths in the $a\times b$ rectangle, and let
$\Dyck_{a,b}$ denote the rational Dyck paths weakly above the
diagonal. Combining diagram tilting with the direct chain-product
realization of the Dyck window gives
\begin{equation}
 \Db{\kk(\ovr A_{a+b}\times\Dyck_{a,b})}
 \simeq\Db{\kk L_{a,b}}.
\label{eq:CLR-conj}
\end{equation}
This is the conjecture of Chapoton--Ladkani--Rognerud
\cite[Conjecture~3.7]{CLR}; its case $a=2$ was proved there.
The window and diagram arguments are established first;
the direct untwisting theorem then gives the tensor-product step.

\begin{corollary}[Corollary~\ref{cor:CLR-conjecture}]
For coprime positive integers $a,b$ and any field $\kk$,
the equivalence \eqref{eq:CLR-conj} holds.
\end{corollary}

The final part of the paper interprets the tilts in Fukaya categories.
The correspondences of \cite{DJL,DiDedda} provide ambient symplectic
models. Corner embeddings select the relevant product Lagrangians
and thimbles, and the relative tilts specify incidence generators
in the resulting subcategories.

\begin{theorem}[Theorems~\ref{thm:staircase-fukaya-corner}
and~\ref{thm:categorical-fukaya-staircase},
Proposition~\ref{prop:staircase-FS-admissible}]

Let $\Omega=[\ell,H]\subseteq\os_m(N)$ be a nonempty coordinate
interval, and put $\overline\Omega=\sigma_N(\Omega)$.
There are enhanced equivalences inducing
\[
 \Db{\kk\Omega}\simeq\mathcal K_\Omega^{\mathrm W}
 \simeq\mathcal K_\Omega^{\mathrm F},
\]
where $\mathcal K_\Omega^{\mathrm W}$ and
$\mathcal K_\Omega^{\mathrm F}$ are the admissible subcategories of
$H^0\Wcat_N^{(m)}$ and $H^0\Fcat(f_{N,m})$ generated respectively
by $L_I$ and $\Delta_I$ for $I\in\overline\Omega$.
Each intermediate category $\C_j(\Omega)$ is quasi-equivalent to
the opposite of the formal enhanced category on a generating collection
$(\mathbb L_v^{(j)})_{v\in\Omega}$, and the relative Kan summands
give the changes of generators. For $m\geq2$, the penultimate
collection $(\mathbb L_v^{(m-2)})$ has embedded product
representatives.
\end{theorem}

Every skew diagram in Theorem~\ref{thm:intro-staircase} has such
an interval model. Theorem~\ref{thm:cubical-geometric-tilting}
gives the cubical presentations of all relative steps, the explicit
product representatives at the last step, and the null-homotopies
of half-square products in the standard Floer-chain model.
For an admissible inclusion $j$, the intrinsic Serre functor is
$j^!S_Aj$. Thus local Dyck periodicity is interpreted by ambient
graded rotation followed by projection at every step.
The symbol $S$ in the abstract denotes this intrinsic Dyck Serre
functor. In the window construction, $S$ denotes the Serre functor
of the full higher Auslander category; its domain is specified there.
Tensor cancellation applied to the coprime CLR equivalence gives
$S^{a+b+1}\cong[(a-1)(b-1)]$.
Proposition~\ref{prop:45-embedded-124} constructs one earlier
degree-three incidence generator by surgery at infinity. Embedded
realizations of the complete earlier collections and a simultaneous
geometric mutation sequence remain open.
Without coprimality, the corner equivalences remain
valid, but Proposition~\ref{prop:noncoprime-rank-obstruction}
rules out the unchanged CLR formula.

Goguet's flip-flop theorems for torsion-class posets related by
$1$-APR tilting \cite[Theorems~1 and~2]{Goguet} provide a useful
comparison. His approach uses torsion classes and silting methods;
the Kan-extension objects here directly control the interlacing
relations and zero composites.

The paper is orgnized as follows. In Section~\ref{sec:preliminaries}, we fix conventions.
In Section~\ref{sec:relative}, we prove relative tilting, and
 develop its coordinate and diagram
forms in Section~\ref{sec:staircase}. In Section~\ref{sec:dyck}, we identify the Dyck diagrams and
corners. In Section~\ref{sec:conjecture}, we prove direct Serre untwisting,
weighted exchange completion and interlacing separation. We derive
the tensor equivalence and the CLR corollary, treats the window's
separate homological consequences, and gives the non-coprime obstruction. In Section~\ref{sec:fukaya}, we develop the
geometric models and in Section~\ref{sec:examples} we make the
constructions explicit.

\section*{Acknowledgments}

This work was supported by the Beijing Natural Science Foundation (Grant Nos. 1262017, 1252011). The author gratefully acknowledges the assistance of ChatGPT-5.6 Sol in computing the examples.

The author thanks those who offered mathematical advice and encouragement,
and his family for their patience and support.

\section{Preliminaries and conventions}
\label{sec:preliminaries}

\subsection{Finite linear categories and modules}

Throughout, $\kk$ is a field. By a finite $\kk$-linear category we
mean a $\kk$-linear category with finitely many objects and with
$\dim_{\kk}\Hom_{\mathcal A}(x,y)<\infty$ for every pair of objects
$x,y$. We distinguish such a category from its category algebra, denoted
$\operatorname{Alg}(\mathcal A)$ below. If $\mathcal A$ is such a category,
$\mathcal A\modcat$ denotes the category of finite-dimensional
covariant $\kk$-linear functors $\mathcal A\longrightarrow \kk\modcat$.
The standard projective and injective modules at $x\in\mathcal A$ are
\[
 P_x=\Hom_{\mathcal A}(x,-),\qquad
 I_x=D\Hom_{\mathcal A}(-,x),
\]
where $D=\Hom_{\kk}(-,\kk)$.  They are indecomposable when
$\End_{\mathcal A}(x)=\kk$.  More generally, the representables are
projective and the corepresentables are injective without any semisimplicity
assumption; they need not be indecomposable when the corresponding diagonal
endomorphism algebra is not local.  We write
$\Db{\mathcal A}=\Db{\mathcal A\modcat}$.

For a finite $\kk$-linear category $\mathcal A$, denote its category
algebra by
\begin{equation*}
 \operatorname{Alg}(\mathcal A)
 =\bigoplus_{u,v\in\operatorname{Ob}\mathcal A} \Hom_{\mathcal A}(u,v),
\end{equation*}
with multiplication induced by composition and with the product of
noncomposable morphisms defined to be zero. Thus an isomorphism with a
finite-dimensional algebra refers to $\operatorname{Alg}(\mathcal A)$,
whereas an isomorphism of linear categories will be stated explicitly at the
categorical level.

\begin{definition}
Let $P$ be a finite poset. Its $\kk$-linear incidence category,
denoted by $\kk P$, has object set $P$ and morphism spaces
\[
\Hom_{\kk P}(x,y)=
\begin{cases}
\kk f_{yx},&x\leq y,\\
0,&x\nleq y.
\end{cases}
\]
Here $f_{yx}:x\to y$ denotes a fixed basis morphism. The identity
morphism of $x$ is $f_{xx}$, and composition is determined by $f_{zy}f_{yx}=f_{zx}$
whenever $x\leq y\leq z$,
together with $\kk$-bilinearity.
Transitivity makes this composition associative: both bracketings
of a composable triple are the basis morphism for its outer pair.
\end{definition}

The category algebra of the incidence category of a finite poset $P$ is
\[
A(P)
:=
\operatorname{Alg}(\kk P)
=
\bigoplus_{x,y\in P}
\operatorname{Hom}_{\kk P}(x,y).
\]
When no confusion can arise, this algebra is also denoted by $\kk P$
and is called the incidence algebra of $P$. It has the $\kk$-basis $\{f_{yx}:x\leq y\}$,
with multiplication $f_{zy}f_{yx}=f_{zx}$ whenever $x\leq y\leq z$,
while the product of two basis morphisms that are not composable is zero.

In the category algebra, matrix rows correspond to target objects
and columns to source objects: $f_{yx}\in e_yA(P)e_x$.

\begin{definition}
A finite $\kk$-linear category $\X$ is called
\emph{homologically directed} if its objects admit a total ordering such that $\Hom_{\X}(x_i,x_j)=0$ for $i>j$,
and $\gldim\End_{\X}(x_i)<\infty$ for every $i$.
The finite-dimensional spaces
$\Hom_{\X}(x_i,x_j)$, $i<j$, may have arbitrary dimension, and
compositions of nonzero morphisms are allowed to be zero.  The category is
called \emph{semisimple-directed} if, in addition,
$\End_{\X}(x_i)$ is semisimple for every $i$.  Thus every
semisimple-directed category is homologically directed.  A
\emph{directed Schur category} is a semisimple-directed category satisfying $\End_{\X}(x)=\kk$ and $\dim_{\kk}\Hom_{\X}(x,x')\leq1$.
\end{definition}

\begin{lemma}
\label{lem:homologically-directed-global-dimension}
Every finite homologically directed category has finite global dimension.  The
same holds for each of its full subcategories.
\end{lemma}

\begin{proof}
Choose a directed ordering $x_1,\ldots,x_n$ and argue by induction
on $n$. For $n=1$, the assertion is the hypothesis on the diagonal
endomorphism algebra. For $n>1$, let $\X'$ be the full subcategory
on $x_1,\ldots,x_{n-1}$ and put $B=\operatorname{Alg}(\X'),
 A_n=\End_{\X}(x_n,
 M=\bigoplus_{i<n}\Hom_{\X}(x_i,x_n)$.
Our multiplication convention gives
$e_{x_j}\operatorname{Alg}(\X)e_{x_i}=\Hom_{\X}(x_i,x_j)$.
Directedness therefore yields the triangular decomposition
\[
 \operatorname{Alg}(\X)\cong
 \begin{pmatrix}B&0\\M&A_n\end{pmatrix},
\]
where postcomposition and precomposition give the left $A_n$-action
and right $B$-action on $M$. By induction, $B$ has finite global
dimension, and the same holds for $A_n$ by hypothesis. The standard
triangular-matrix criterion \cite[Corollary~6.3]{MinamotoYamaura}
proves the assertion. Every full subcategory inherits the directed
ordering and the same diagonal algebras, so the argument applies
to it as well.
\end{proof}

\subsection{Sieves}

\begin{definition}
Let $\X$ be a linear category.  A full subcategory $\mathcal S\subseteq\X$
is a \emph{sieve} if, whenever $s\in\mathcal S$ and
$\Hom_{\X}(x,s)\neq0$, one has $x\in\mathcal S$.
\end{definition}

For an ordinary poset regarded as a category, a full sieve is simply a lower
order ideal.  In a category with zero compositions, the definition only uses
the existence of individual nonzero morphisms.

\subsection{Tensor product categories}

If $Y$ is a finite poset, write $\kk Y$ for its incidence category.
The tensor product $\X\otimes\kk Y$ has objects $(x,y)$ and $\Hom_{\X\otimes\kk Y}((x,y),(x',y'))=\Hom_{\X}(x,x')\otimes_{\kk}\Hom_{\kk Y}(y,y')$.
Thus this space is $\Hom_{\X}(x,x')$ when $y\leq y'$, and zero otherwise.

\subsection{Increasing sequences and higher Auslander categories}

For a positive integer $N$, write $[N]=\{1,2,\ldots,N\}$.
For positive integers $m$ and $N$, define $\os_m(N)
=
\{(z_1,\ldots,z_m):
1\leq z_1<\cdots<z_m\leq N\}$.
Thus $\os_m(N)$ is the set of strictly increasing sequences of length
$m$ with entries in $[N]$. Equivalently, it is the set of $m$-element
subsets of $[N]$, written in increasing order. Hence $|\os_m(N)|=\binom{N}{m}$
when $1\leq m\leq N$, while $\os_m(N)=\varnothing$ when $m>N$.

For $x=(x_1,\ldots,x_m),
y=(y_1,\ldots,y_m)$
in $\os_m(N)$, write
\begin{equation}
x\preceq y
\quad\Longleftrightarrow\quad
x_1\leq y_1<x_2\leq y_2<\cdots<x_m\leq y_m.
\label{eq:interlace}
\end{equation}
Thus $x\preceq y$ means that the coordinates of $x$ and $y$
interlace. Equivalently, it requires $x_i\leq y_i$ for $1\leq i\leq m$,
together with $y_i<x_{i+1}$ for $1\leq i<m$.

This relation is not transitive: $13\preceq24\preceq35$
but $13\npreceq35$. Its failure of transitivity produces the
zero composites in the higher Auslander category.

\begin{remark}
The notation used here is the strict-coordinate form of the
ordered-sequence notation of Jasso and K\"ulshammer
\cite[Definition~1.9]{JassoKulshammer}. For a poset $P$, they write $\os^{m}(P)
=\{(u_1,\ldots,u_m)\in P^m:
u_1\leq\cdots\leq u_m\}$.
When $1\leq m\leq N$, the coordinate shift $z_i=u_i+i-1$ for $1\leq i\leq m$
induces a bijection $\os^{m}([N-m+1])
\xrightarrow{\sim}
\os_m(N)$.
Its inverse is given by $u_i=z_i-i+1$.
Under this bijection, the weak interlacing condition $u_1\leq v_1\leq u_2\leq v_2
\leq\cdots\leq u_m\leq v_m$
becomes the strict-coordinate condition $z_1\leq w_1<z_2\leq w_2
<\cdots<z_m\leq w_m$
used in \eqref{eq:interlace}.
\end{remark}

\begin{definition}
The higher Auslander category of type $A$ with vertex set $\os_m(N)$
has morphism spaces
\begin{equation*}
\Hom(x,y)
=
\begin{cases}
\kk f_{yx}, & x\preceq y,\\
0,              & x\npreceq y,
\end{cases}
\end{equation*}
where $f_{yx}\colon x\to y$ is a fixed basis morphism. For composable
basis morphisms $x\xrightarrow{f_{yx}}y\xrightarrow{f_{zy}}z$,
composition is defined by
\begin{equation}
f_{zy}f_{yx}
=
\begin{cases}
f_{zx}, & x\preceq z,\\
0,      & x\npreceq z.
\end{cases}
\label{eq:higher-composition}
\end{equation}
In particular, for the three vertices $x=(1,3), y=(2,4),z=(3,5)$
considered above, both $f_{yx}$ and $f_{zy}$ are nonzero, but $f_{zy}f_{yx}=0$
because $x\npreceq z$.
\end{definition}

Equivalently, this category is the path category of the quiver whose
arrows increment one coordinate, modulo the commutativity relations
associated with complete squares and the zero relations associated with
half-squares.
The higher Auslander algebras of type $A$ arise from Iyama's higher
Auslander--Reiten theory and higher-dimensional Auslander correspondence;
see \cite{Iyama2007,Iyama2011}. Their combinatorial description in terms of
interlacing sequences was developed by Oppermann and Thomas; see
\cite{OppermannThomas,JassoKulshammer}. Our notation is
$A_n^m$, with vertex set $\os_m(n+m-1)$; this is the
$(m-1)$-Auslander algebra of type $A_n$ in the convention of
\cite[Section~4.1]{Xing}.

\begin{remark}
\label{rem:index-conversion}
Gottesman's $d$-Auslander algebra has vertices indexed by increasing
sequences of length $d+1$ \cite{Gottesman}. Consequently,
Gottesman's $A_n^{m-1}$ is our $A_n^m$. We will display this index
conversion whenever the conventions of the two sources are used
together.
\end{remark}

\subsection{Standard cluster-tilting and Serre-period inputs}
\label{subsec:cluster-preliminaries}

This subsection recalls known results used only for the window's
homological consequences. The direct tensor equivalence and the CLR
proof do not use them. All modules here are right modules.

A finite-dimensional module $M$ is $N$-cluster tilting when
\[
\begin{split}
 \add M
 &=\{Z\in B\modcat:
       \Ext_B^i(M,Z)=0\ (1\leq i<N)\}\\
 &=\{Z\in B\modcat:
    \Ext_B^i(Z,M)=0\ (1\leq i<N)\}.
\end{split}
\]
It is $N\mathbb Z$-cluster tilting if its self-Ext groups also
vanish in positive degrees outside $N\mathbb Z$.
An algebra is $N$-representation finite when it has such an
$N$-cluster tilting module and global dimension at most $N$.
In a triangulated category, cluster tilting includes functorial
finiteness and the analogous two Hom-orthogonality conditions;
the $N\mathbb Z$ version is $[N]$-closed and has no Hom in degrees
outside $N\mathbb Z$.

\begin{proposition}
\label{prop:standard-window-inputs}
\begin{enumerate}
\item Let $B$ be a finite-dimensional algebra of global dimension
at most $N$, and set
\[
 \mathcal U_N(B)=\add\{(S_B[-N])^jB:j\in\mathbb Z\}.
\]
Then $DB\in\mathcal U_N(B)$ if and only if $B$ is
$N$-representation finite. In that case $\mathcal U_N(B)$ is
$N$-cluster tilting, and its unique module part $\mathcal M$ satisfies
\[
 \mathcal M=\mathcal U_N(B)\cap B\modcat,
 \qquad
 \mathcal U_N(B)=\add\{M[Nj]:M\in\mathcal M,\ j\in\mathbb Z\}.
\]
\item For $d,n\geq1$, the Serre functor of
$\Db{A_{n+1}^{d}\modcat}$ satisfies
the natural triangle-functor isomorphism
\begin{equation}
 S^{d+n+1}\cong[dn].
\label{eq:full-algebra-period}
\end{equation}
\end{enumerate}
\end{proposition}

Part (1) is \cite[Theorems~2.6 and~2.5]{IyamaOppermann},
applied in that order, together with
\cite[Theorem~1.6 and Lemma~5.5(b)]{Iyama2011} for the module part.
The first theorem supplies the finiteness required by the second;
passing to the opposite algebra gives the stated module convention.
Part (2) is the known ambient Serre period: under
$A_{n+1}^{d}=\mathsf A_{d+n,d}$, graded disk rotation gives
$S^{r+1}\cong[d(r-d)]$ for $r=d+n$
\cite[Proposition~3.5.1 and Remark~3.5.2]{DJL}.
Neither assertion concerns a Dyck corner or depends on CLR.

\section{A relative interval-tilting theorem}\label{sec:relative}

\subsection{The two categories}

Let $\X$ be a finite $\kk$-linear category, and let $Y$ be a
finite poset. Denote by $\operatorname{Sieve}(\X)$ the poset of full
sieves of $\X$, ordered by inclusion. Let
\[
F\colon Y\longrightarrow \operatorname{Sieve}(\X),
\qquad
y\longmapsto F(y),
\]
be an order-preserving map. Thus, for every $y\in Y$, the category
$F(y)$ is a full sieve of $\X$, and
\begin{equation*}
y\leq y'
\quad\Longrightarrow\quad
F(y)\subseteq F(y')
\qquad
\text{for all }y,y'\in Y.
\end{equation*}

\begin{definition}
Let $\G=\G(\X,Y,F)$ be the full subcategory of
$\X\otimes\kk Y$ whose object set is $\Ob\G
=
\{(x,y):y\in Y,\ x\in\Ob F(y)\}$.
For any two objects $(x,y),(x',y')\in\operatorname{Ob}\G$,
that is, for $y,y'\in Y$, $x\in F(y)$, and $x'\in F(y')$, the
morphism space is
\[
\Hom_{\G}((x,y),(x',y'))
=
\begin{cases}
\Hom_{\X}(x,x'), & y\leq y',\\
0,                & y\nleq y'.
\end{cases}
\]
Composition is induced from the tensor-product category
$\X\otimes\kk Y$. Explicitly, if $(x,y)\xrightarrow{f}(x',y')
\xrightarrow{g}(x'',y'')$
are composable morphisms, then $y\leq y'\leq y''$, and their composite
is the morphism $gf\in\Hom_{\X}(x,x'')$
viewed as a morphism from $(x,y)$ to $(x'',y'')$ in $\G$.
\end{definition}

The second category retains a morphism only when its target lies in
the fiber of its source.

\begin{definition}
The category $\Gs=\Gs(\X,Y,F)$ has the same objects as $\G$, and
\begin{equation*}
 \Hom_{\Gs}((x,y),(x',y'))=
 \begin{cases}
 \Hom_{\X}(x,x'),&y\leq y'\text{ and }x'\in F(y),\\
 0,&\text{otherwise}.
 \end{cases}
\end{equation*}
For three objects $(x_0,y_0),\ (x_1,y_1),\ (x_2,y_2)\in\operatorname{Ob}\Gs$
and composable morphisms
$(x_0,y_0)\xrightarrow{f}(x_1,y_1)
\xrightarrow{g}(x_2,y_2)$,
define
\begin{equation}
g\star f=
\begin{cases}
gf, & x_2\in\operatorname{Ob}F(y_0),\\
0,  & x_2\notin\operatorname{Ob}F(y_0).
\end{cases}
\label{eq:sharp-composition}
\end{equation}
Here $F(y_0)$ is the full sieve of $\X$ assigned by $F$ to the
source layer $y_0$, and $x_2$ is the $\X$-component of the final
target object. Thus the first case means that the final target lies
in the sieve associated with the initial source layer. The product
$gf$ in the first line is the ordinary composite in $\X$ and is
understood to be zero if that composite vanishes in $\X$.
\end{definition}

\begin{lemma}\label{lem:associativity}
Let $\mathcal I$ consist of the morphism spaces of $\G$ whose final
$\X$-object lies outside the source fiber:
\[
 \mathcal I((x,y),(x',y'))=
 \begin{cases}
 \Hom_{\G}((x,y),(x',y')),&x'\notin F(y),\\
 0,&x'\in F(y).
 \end{cases}
\]
Then $\mathcal I$ is a two-sided ideal and
$\Gs\cong\G/\mathcal I$. In particular, the product
\eqref{eq:sharp-composition} is associative. If $Y\ne\varnothing$,
then $\mathcal I^{|Y|}=0$.
\end{lemma}
\begin{proof}
Precomposition moves the source to a smaller fiber, so a target
outside the old fiber remains outside the new one. For
postcomposition, a nonzero map from an object outside $F(y)$
to an object inside $F(y)$ is excluded by the sieve property.
Thus both sides preserve $\mathcal I$. The quotient has exactly
the morphisms and multiplication of $\Gs$, proving the assertion.
Every morphism in $\mathcal I$ strictly increases the $Y$-layer,
so a product of $|Y|$ such morphisms is zero.
\end{proof}

\subsection{Restriction and exact right Kan extension}

For $y\in Y$, let
\[
 \iota_y:F(y)\longrightarrow\G,\qquad x\longmapsto(x,y).
\]
Precomposition with $\iota_y$ gives the exact restriction functor
\[
 \iota_y^*:\G\modcat\longrightarrow F(y)\modcat.
\]

\begin{lemma}\label{lem:kan}
For each $y\in Y$, the functor $\iota_y^*$ has an exact right adjoint $(\iota_y)_*: F(y)\modcat\to \G\modcat$
given on objects by
\begin{equation}
 ((\iota_y)_*M)(a,b)=
 \begin{cases}
 M(a),&b\leq y,\\
 0,&b\nleq y.
 \end{cases}
\label{eq:kan-formula}
\end{equation}
\end{lemma}

\begin{proof}
The pointwise formula for the $\kk$-linear right Kan extension is
\[
 ((\iota_y)_*M)(a,b)
 =\operatorname{Nat}_{F(y)}
 \bigl(\Hom_{\G}((a,b),\iota_y-),M\bigr).
\]
If $b\leq y$, monotonicity gives $a\in F(y)$ and fullness
identifies the first argument with $P_a^{F(y)}$; linear Yoneda
gives $M(a)$. If $b\nleq y$, that argument is the zero functor,
so the value is zero. This proves \eqref{eq:kan-formula} and
the adjunction by the universal property of right Kan extension.
On a morphism with $\X$-component $u$ and final layer $d$, the
induced map is $M(u)$ when $d\leq y$, and zero otherwise.
Finally, evaluation of $(\iota_y)_*$ is either evaluation in $F(y)$
or zero. Kernels and cokernels are pointwise, so the functor is exact.
\end{proof}

\begin{lemma}\label{lem:restriction-kan}
Let $y,y'\in Y$ and $x\in F(y)$, and write
$P_x^{F(y)}=\Hom_{F(y)}(x,-)$.
Then the restriction is representable or zero:
\begin{equation*}
\iota_{y'}^*(\iota_y)_*P_x^{F(y)}
\cong
\begin{cases}
P_x^{F(y')},
  & y'\leq y\text{ and }x\in F(y'),\\
0,
  & \text{otherwise}.
\end{cases}
\end{equation*}
\end{lemma}

\begin{proof}
Formula \eqref{eq:kan-formula} makes the restriction zero unless
$y'\leq y$. In that case its value at $a\in F(y')$ is
$\Hom_{\X}(x,a)$. If $x\in F(y')$, fullness identifies this
with $P_x^{F(y')}$, including its structure maps.
Otherwise every such Hom space is zero by the sieve property.
This proves both formulations.
\end{proof}

\subsection{The tilting object}

For $x\in F(y)$, put
\begin{equation}
 T_{x,y}=(\iota_y)_*P_x,\qquad
 T=\bigoplus_{y\in Y}\ \bigoplus_{x\in F(y)}T_{x,y}.
\label{eq:T-relative}
\end{equation}

\begin{proposition}\label{prop:T-rigid}
For every integer $p$, the Kan summands satisfy
\begin{equation}
\begin{aligned}
 &\Hom_{\Db{\G}}(T_{x,y},T_{x',y'}[p])\\
 &\quad\cong
 \begin{cases}
 \Hom_{\X}(x',x),&p=0,\ y'\leq y,\ x\in F(y'),\\
 0,&\text{otherwise}.
 \end{cases}
\end{aligned}
\label{eq:rigidity-derived-adjunction}
\end{equation}
In particular, $\Hom_{\Db{\G}}(T,T[p])=0$ for $p\ne0$.
\end{proposition}
\begin{proof}
Because both functors in Lemma~\ref{lem:kan} are exact, their
adjunction descends to bounded derived categories and identifies
the left-hand side with $\Hom_{\Db{F(y')}}(\iota_{y'}^*T_{x,y},P_{x'}^{F(y')}[p])$.
By Lemma~\ref{lem:restriction-kan}, its first argument is either
$P_x^{F(y')}$ or zero. In the former case projectivity concentrates
derived Hom in degree zero, and Yoneda gives $\Hom_{\X}(x',x)$.
Taking the finite direct sum proves the rigidity of $T$.
\end{proof}

\begin{proposition}
\label{prop:T-generates}
The module $T$ is perfect and
$\thick(T)=\Kb{\pmodcat{\G}}$, without any global-dimension
assumption. If $Y\ne\varnothing$, then
$\operatorname{pd}_{\G}T\leq |Y|-1$.
\end{proposition}

\begin{proof}
Write $Y=\{y_1,\ldots,y_q\}$ in a linear extension. For
$x\in F(y_i)$, let $J_{x,i}$ be the module equal to
$P_x^{F(y_i)}$ on layer $y_i$ and zero elsewhere.
A nondecreasing sequence of layers cannot leave and return to
$y_i$, so this is a module.

Every module has a finite filtration by its restrictions to
individual layers: its values on the upper set
$\{y_i,\ldots,y_q\}$, extended by zero, define a submodule.
For $Q_{x,i}=\Hom_{\G}((x,y_i),-)$, this gives an exact sequence
\[
 0\longrightarrow K_{x,i}\longrightarrow Q_{x,i}
 \longrightarrow J_{x,i}\longrightarrow0.
\]
The kernel has a layer filtration with factors $J_{x,j}$ for
$y_i<y_j$. At the last layer $J_{x,q}=Q_{x,q}$ is projective.
Descending induction therefore proves that $J_{x,i}$
is perfect with projective dimension at most $q-i$.
The restrictions of $T_{x,y_i}$ are representables or zero by
Lemma~\ref{lem:restriction-kan}. Its layer factors are thus among
these $J_{x,j}$, proving perfectness and the bound $q-1$.

The same filtration proves generation, now by ascending induction. The module $J_{x,i}$
is a submodule of $T_{x,y_i}$ whose quotient is filtered by
$J_{x,j}$ with $j<i$. Hence every $J_{x,i}$ belongs to $\thick(T)$.
The same layer filtration puts every $Q_{x,i}$ there.
These representables generate $\Kb{\pmodcat{\G}}$, which proves
the equality. The case $Y=\varnothing$ is immediate.
\end{proof}

\subsection{The endomorphism category}

For an indexed family $(T_i)_{i\in I}$, its endomorphism
category has objects $I$ and morphisms $\Hom(T_i,T_j)$, with
composition inherited from the ambient category. Its category
algebra is $\End(\bigoplus_iT_i)$. We keep the index set even
when some objects are isomorphic.

For the indexed family $T_{x,y}=(\iota_y)_*P_x^{F(y)}, y\in Y,\quad x\in F(y)$,
of summands of the tilting object $T$ in \eqref{eq:T-relative}, let
\[
 \mathcal E_T
 =
 \operatorname{EndCat}_{\Db{\G}}
 \{T_{x,y}:y\in Y,\ x\in F(y)\}.
\]
Thus the labels $(x,y)$ are retained as the objects of $\mathcal E_T$, and
its morphisms and composition are those in $\Db{\G}$.

\begin{proposition}\label{prop:end-relative}
There is an isomorphism of linear categories
\[
 \mathcal E_T^{\op}\cong\Gs.
\]
Consequently, taking category algebras gives an isomorphism of algebras
\[
 \End_{\Db{\G}}(T)^{\op}
 \cong\operatorname{Alg}(\Gs).
\]
\end{proposition}

\begin{proof}
For a morphism $f:(x,y)\to(x',y')$ in $\G$, define
$u_f:T_{x',y'}\to T_{x,y}$ by precomposition:
\begin{equation*}
 (u_f)_{(a,b)}:
 \Hom_{\X}(x',a)\longrightarrow\Hom_{\X}(x,a),
 \qquad h\longmapsto hf
 \quad(b\leq y),
\end{equation*}
and by zero on the other layers. The maps commute with the module structure maps. For composable
$f,g$, their components send $h$ to $hgf$ on the layers where
the target is nonzero; on every other layer both sides vanish.
Thus $u_fu_g=u_{gf}$, including zero composites.
This defines a functor $\G\to\mathcal E_T^{\op}$.

If $x'\notin F(y)$, then on every layer $b\leq y$ the source
$\Hom_{\X}(x',a)$ is zero by the sieve property.
Hence this functor kills $\mathcal I$ of Lemma~\ref{lem:associativity}.
Formula \eqref{eq:rigidity-derived-adjunction}, with the two
arguments reversed, identifies its induced Hom maps on
$\G/\mathcal I$ with Yoneda isomorphisms.
It is the identity on the indexed objects, so
$\Gs\cong\G/\mathcal I\cong\mathcal E_T^{\op}$.
Passing to category algebras gives the endomorphism-algebra assertion.
\end{proof}

\begin{theorem}\label{thm:relative}
\label{thm:relative-without-gldim}
Let $\X$ be a finite $\kk$-linear category, $Y$ a finite poset,
and $F(y)$ an increasing family of full sieves. The Kan-extension
module $T$ in \eqref{eq:T-relative} is a tilting complex, and
\[
 \mathcal E_T^{\op}\cong\Gs,\qquad
 \End_{\Db{\G}}(T)^{\op}\cong\operatorname{Alg}(\Gs),\qquad
 \Db{\G}\simeq\Db{\Gs}.
\]
The conclusion holds for arbitrary global dimension. If
$Y\ne\varnothing$, then $\operatorname{pd}_{\G}T\leq |Y|-1$.
\end{theorem}

\begin{proof}
Proposition~\ref{prop:T-rigid} gives rigidity,
Proposition~\ref{prop:T-generates} gives perfect generation, and
Proposition~\ref{prop:end-relative} identifies the endomorphism
algebra, including its multiplication. Rickard's derived Morita theorem
\cite{Rickard1989} gives the equivalence.
\end{proof}

\begin{corollary}
\label{cor:directed-relative}
In the setting of Theorem \ref{thm:relative}, suppose that the objects of $\X$
admit an ordering $x_1,\ldots,x_n$ such that $\Hom_{\X}(x_i,x_j)=0$ for $i>j$,
and $\gldim\End_{\X}(x_i)<\infty$ for every $i$.
Then $\G(\X,Y,F)$ and every $F(y)$ have finite global dimension.
In this case $\Kb{\pmodcat{\G}}=\Db{\G}$, so the tilting
module also generates the whole bounded derived category.
\end{corollary}

\begin{proof}
The hypotheses say exactly that $\X$ is homologically directed.  Every full
fiber $F(y)$ inherits the same property.  Moreover, the product of the
given ordering of $\X$ with a linear extension of $Y$ gives a directed
ordering of $\G$, and $\End_{\G}(x,y)=\End_{\X}(x)$.
Thus $\G$ is homologically directed as well.  Applying
Lemma \ref{lem:homologically-directed-global-dimension} to $\G$ and to every
$F(y)$ proves the assertions; the final equality follows from
finite global dimension.
\end{proof}

\begin{remark}
When $\X$ is the incidence category of a finite poset, this construction is
the mechanism of \cite[Theorem~3.1]{CLR}.  The extension needed here is that
$\X$ may have arbitrary finite-dimensional Hom spaces, non-semisimple
diagonal endomorphism algebras, and its own zero compositions.  Moreover,
closed subsets of a poset are replaced by full sieves.
The boundary-ideal quotient describes the target multiplication,
derived Yoneda gives rigidity, and the layer filtration supplies
perfect generation. Thus the relative tilt is controlled by a
finite directed filtration rather than by diagonal semisimplicity
or finite global dimension. The projectives $Q_{x,i}$ extend toward
later layers, while the Kan modules $T_{x,y_i}$ extend toward
earlier ones. Both are assembled from the same layer projectives
$J_{x,i}$, and the tilt changes between these two systems of
generators.
\end{remark}

\begin{remark}
At infinite global dimension, $T$ generates the perfect category
$\Kb{\pmodcat{\G}}$, as required by Rickard's criterion.
This need not equal $\Db{\G}$. The
result also induces an equivalence of singularity categories $\Db{\G}/\Kb{\pmodcat{\G}}
 \simeq \Db{\Gs}/\Kb{\pmodcat{\Gs}}$,
since a standard derived equivalence preserves perfect complexes.
\end{remark}

\begin{exm}
Let $R$ be any finite-dimensional $\kk$-algebra and $M$ a
nonzero finite-dimensional right $R$-module. Let $\X$ have
objects $0,1$, with $\End(0)=R$, $\End(1)=\kk$,
$\Hom(0,1)=M$, and $\Hom(1,0)=0$; composition uses the
right $R$-action on $M$. Take $Y=\{u<v\}$,
$F(u)=\{0\}$, and $F(v)=\X$, and write
$a=(0,u)$, $b=(0,v)$, $c=(1,v)$.
The vertical identity $a\to b$ induces
\[
 0\longrightarrow P_b\longrightarrow P_a
 \longrightarrow U_a\longrightarrow0,
\]
where $U_a$ is supported at $a$ with value $R$.
The Kan summands are $U_a$, $P_a$, and $P_c$, so
$T\simeq[P_b\longrightarrow P_a]\oplus P_a\oplus P_c$,
with the two-term complex in degrees $-1,0$.
Theorem~\ref{thm:relative-without-gldim} gives
$\Db{\G}\simeq\Db{\Gs}$: the morphism space from $a$ to
$c$ is $M$ in $\G$ and zero in $\Gs$. No condition on
$\gldim R$ or projectivity of $M$ is needed. For example,
$R=\kk[\epsilon]/(\epsilon^2)$ gives an infinite-global-dimension
diagonal, while $R=\kk$ and $M=\kk^q$ exhibit arbitrary
morphism dimensions. The indexed representables may be
decomposable, as allowed by the theorem.
\end{exm}

\section{Coordinate intervals and horizontal-strip categories}
\label{sec:staircase}

\subsection{Coordinates and intermediate categories}

Fix $m\geq1$, $N\geq m$, and $\ell,H\in\os_m(N)$ with
$\ell\leq H$. Put
\[
 \Omega=\Omega(\ell,H)
 =\{z\in\os_m(N):\ell_i\leq z_i\leq H_i\text{ for all }i\},
\]
ordered coordinatewise. A coordinate staircase is the special case
$\ell=(1,\ldots,m)$:
\begin{equation*}
 \Omega(H)=\{1\leq z_1<\cdots<z_m:z_i\leq H_i\}.
\end{equation*}
For arbitrary prescribed upper bounds, replace $H_i$ by
$\min_{r\geq i}(H_r-r+i)$. This preserves the allowed tuples
and makes the effective bounds strictly increasing. A nonempty
staircase is therefore a principal lower interval.

\begin{definition}
\label{def:Cj}
For $0\leq j\leq m-1$, define a linear category
$\C_j(\Omega)$ with object set $\Omega$ and
\begin{equation*}
 \Hom_{\C_j}(x,y)=
 \begin{cases}
 \kk f_{yx},&
 x_i\leq y_i\text{ for all }i,\quad
 y_i<x_{i+1}\text{ for }1\leq i\leq j,\\
 0,&\text{otherwise}.
 \end{cases}
\end{equation*}
For composable basis morphisms, set
\begin{equation*}
 f_{zy}f_{yx}=
 \begin{cases}
 f_{zx},&\Hom_{\C_j}(x,z)\neq0,\\
 0,&\Hom_{\C_j}(x,z)=0.
 \end{cases}
\end{equation*}
\end{definition}

We verify associativity directly. Consider three composable nonzero basis
morphisms $x^0\longrightarrow x^1\longrightarrow x^2\longrightarrow x^3$.
If $\Hom_{\C_j}(x^0,x^3)\neq0$, then, for every $i\leq j$, $x_i^2\leq x_i^3<x_{i+1}^0$ and $x_i^3<x_{i+1}^0\leq x_{i+1}^1$.
Thus both intermediate outer pairs $(x^0,x^2)$ and $(x^1,x^3)$
satisfy the required inequalities, and both bracketings are the basis
morphism from $x^0$ to $x^3$.  If
$\Hom_{\C_j}(x^0,x^3)=0$, the last multiplication in either bracketing is
zero by the outer-pair rule. Hence composition is associative.
At $j=0$, the category $\C_0(\Omega)=\kk\Omega$ is the
incidence category; at $j=m-1$, it has full interlacing multiplication.

\subsection{Adding one interlacing inequality}

Fix $1\leq j\leq m-1$.  Write an element of $\Omega$ as
$z=(p,t),p=(z_1,\ldots,z_j),t=(z_{j+1},\ldots,z_m)$.
Define the prefix and tail sets as the corresponding coordinate projections
of $\Omega$:
\begin{equation*}
 X_j
 =
 \left\{
 (z_1,\ldots,z_j):
 (z_1,\ldots,z_m)\in\Omega
 \text{ for some }z_{j+1},\ldots,z_m
 \right\},
\end{equation*}
and
\begin{equation*}
 Y_j
 =
 \left\{
 (z_{j+1},\ldots,z_m):
 (z_1,\ldots,z_m)\in\Omega
 \text{ for some }z_1,\ldots,z_j
 \right\}.
\end{equation*}
We equip $Y_j$ with the coordinatewise order. Since $\Omega$ is finite,
both $X_j$ and $Y_j$ are finite. Equivalently, an element
$t=(t_1,\ldots,t_{m-j})$ belongs to $Y_j$ precisely when
\[
 t_1<\cdots<t_{m-j},
 \qquad
 \ell_{j+q}\leq t_q\leq H_{j+q}\quad(1\leq q\leq m-j),
\]
and there exists $p\in X_j$ such that $p_j<t_1$.
Define the fully interlaced prefix category $\X_j$ as follows.  Its
objects are the elements of $X_j$, and
\begin{equation*}
 \Hom_{\X_j}(p,p')=
 \begin{cases}
  \kk f_{p'p},&
  p_i\leq p'_i\text{ for all }i,\quad
  p'_i<p_{i+1}\text{ for }1\leq i<j,\\
  0,&\text{otherwise}.
 \end{cases}
\end{equation*}
The identity of $p$ is $f_{pp}$.  For composable basis morphisms, put
\begin{equation}
 f_{p''p'}f_{p'p}=
 \begin{cases}
  f_{p''p},&\Hom_{\X_j}(p,p'')\neq0,\\
  0,&\Hom_{\X_j}(p,p'')=0,
 \end{cases}
\label{eq:Xj-composition}
\end{equation}
and extend composition $\kk$-bilinearly.  The associativity argument after Definition~\ref{def:Cj} applies
to these length-$j$ tuples and proves that
\eqref{eq:Xj-composition} is associative. Thus $\X_j$ is a
directed Schur category with the specified zero composites.

For $t\in Y_j$, define
 $F_j(t)=
 \{p\in X_j:p_j<t_1\}$.

\begin{lemma}\label{lem:Fj}
The assignment $t\mapsto F_j(t)$ is order preserving and each $F_j(t)$
is a full sieve in $\X_j$.
\end{lemma}

\begin{proof}
If $t\leq t'$, then $p_j<t_1\leq t'_1$, proving
$F_j(t)\subseteq F_j(t')$.
Suppose $q\to p$ is nonzero in $\X_j$ and $p\in F_j(t)$.  In particular
$q_j\leq p_j<t_1$,
so $q\in F_j(t)$.  Hence $F_j(t)$ is a sieve.
\end{proof}

\begin{proposition}\label{prop:one-step}
For $1\leq j\leq m-1$, we have
\[
 \Db{\C_{j-1}(\Omega)}\simeq\Db{\C_j(\Omega)}.
\]
\end{proposition}

\begin{proof}
The independent coordinate bounds give
\[
 (p,t)\in\Omega
 \quad\Longleftrightarrow\quad
 p\in X_j,\quad t\in Y_j,\quad p_j<t_1.
\]
Membership in the projections supplies every coordinate bound and
every internal strict inequality; the boundary inequality is the
only remaining condition. Lemma~\ref{lem:Fj} therefore gives a
monotone family of full sieves.

Apply Theorem~\ref{thm:relative-without-gldim} to
$\X_j,Y_j,F_j$. The displayed factorization identifies the objects
of $\G$ with $\Omega$. Its morphism spaces impose exactly the inequalities defining
$\C_{j-1}$. These identifications preserve
composition: the tail incidence factor has no additional zero
relations, and the prefix factor vanishes precisely when an outer
interlacing inequality with index $i<j$ fails. Thus
$\G\cong\C_{j-1}$ as linear categories.

In $\Gs$, the additional condition $p'\in F_j(t)$ is
$p'_j<t_1$, which is the new inequality $y_j<x_{j+1}$.
The outer-pair rule supplies exactly the multiplication of $\C_j$,
so $\Gs\cong\C_j$. The relative theorem proves the equivalence.
\end{proof}

\begin{theorem}
\label{thm:staircase}
\label{thm:coordinate-intervals}
For every coordinate interval $\Omega=\Omega(\ell,H)$,
\begin{equation}
 \Db{\kk\Omega}\simeq\Db{\C_{m-1}(\Omega)}.
\label{eq:staircase-equivalence}
\end{equation}
More precisely, \eqref{eq:staircase-equivalence} is the composite of
$m-1$ explicit tilting equivalences
\[
 \C_0\simeq_{\mathrm{der}}\C_1
 \simeq_{\mathrm{der}}\cdots
 \simeq_{\mathrm{der}}\C_{m-1}.
\]
\end{theorem}

\begin{proof}
The identity $\C_0=\kk\Omega$ and
Proposition \ref{prop:one-step} give the chain of equivalences.
\end{proof}

\subsection{Corners and horizontal-strip categories}

Put $A=A_{N-m+1}^{m}$ and
$e_\Omega=\sum_{z\in\Omega}e_z$.

\begin{proposition}
\label{prop:corner}
There is an isomorphism of algebras
\[
 \operatorname{Alg}(\C_{m-1}(\Omega))\cong e_\Omega Ae_\Omega.
\]
Consequently,
$\Db{\kk\Omega}\simeq\Db{e_\Omega Ae_\Omega}$.
\end{proposition}

\begin{proof}
The Hom condition is precisely the higher Auslander interlacing
condition, and the outer-pair composition rules agree. Hence the
last category is the full subcategory on the selected vertices.
Moreover, a monotone quiver path whose endpoints lie in $[\ell,H]$
stays in that interval.
\end{proof}

The coordinate statement has an intrinsic diagrammatic form.
Give the cells of a Ferrers diagram the product order:
$(r,c)\leq(r',c')$ if $r\leq r'$ and $c\leq c'$.
For partitions $\mu\subseteq\lambda$, let
$\Theta=\lambda\setminus\mu$ have the induced cell order.
For a finite poset $Q$, write $J(Q)$ for its order ideals, ordered
by inclusion, and put $J_{m,n}=J(\ovr A_m\times\ovr A_n)$.
A horizontal strip is a skew diagram containing at most one cell
in each column.

\begin{theorem}
\label{thm:skew-strip-tilting}
Let $\mu\subseteq\lambda$ be partitions in an $m\times n$
rectangle, with $m,n\geq1$. Define $\mathcal H_{\lambda/\mu}$
as follows. Its objects are the intermediate partitions
$\mu\subseteq\nu\subseteq\lambda$. There is one basis morphism
$u_{\eta\nu}:\nu\to\eta$ precisely when $\nu\subseteq\eta$
and $\eta/\nu$ is a horizontal strip. The composite
$u_{\zeta\eta}u_{\eta\nu}$ is $u_{\zeta\nu}$ when
$\zeta/\nu$ is a horizontal strip, and is zero otherwise.
Then this defines a finite linear category, and
\[
 \Db{\kk J(\Theta)}
 \simeq\Db{\mathcal H_{\lambda/\mu}}
 \simeq\Db{e_{\lambda/\mu}A_{n+1}^{m}e_{\lambda/\mu}},
\]
where $e_{\lambda/\mu}$ selects the labels $z_i=i+\nu_{m+1-i} (1\leq i\leq m)$
of the intermediate partitions. By transposition, we also have
$\Db{\mathcal H_{\lambda/\mu}}
 \simeq\Db{\mathcal H_{\lambda^T/\mu^T}}$,
the latter having a corner realization in $A_{m+1}^{n}$.
\end{theorem}

\begin{proof}
The map $\nu\mapsto\nu\setminus\mu$ identifies the intermediate
partitions with $J(\Theta)$: an ideal of the induced difference,
together with $\mu$, is an ideal of $\lambda$, and the converse
is immediate. The increasing-coordinate map identifies this
interval with $\Omega(\ell,H)\subseteq\os_m(m+n)$, where
\[
 \ell_i=i+\mu_{m+1-i},\qquad
 H_i=i+\lambda_{m+1-i}.
\]
If $z,w$ correspond to $\nu,\eta$, then $z\leq w$ means
$\nu\subseteq\eta$, and
\[
 w_i<z_{i+1}
 \quad\Longleftrightarrow\quad
 \eta_r\leq\nu_{r-1}\quad(r=m+1-i,\ 2\leq r\leq m).
\]
These inequalities are exactly the horizontal-strip condition.
The outer-pair multiplication agrees as well. Thus
$\mathcal H_{\lambda/\mu}\cong\C_{m-1}(\Omega(\ell,H))$;
this proves associativity and the corner identification.
Theorem~\ref{thm:coordinate-intervals} proves the first equivalence.
Finally, cell transposition is an isomorphism
$\Theta\cong\lambda^T\setminus\mu^T$, hence an isomorphism of
their ideal lattices. Apply the first part to both diagrams.
\end{proof}

Equivalently, $\mathcal H_{\lambda/\mu}$ has the ideals of $\Theta$
as objects; a basis morphism $I\to J$ exists when $I\subseteq J$
and $J\setminus I$ contains at most one cell in each column.
Thus this presentation uses only the diagram and its columns.
Replacing columns by rows defines its vertical-strip category.
Transposition identifies that category with
$\mathcal H_{\lambda^T/\mu^T}$, so the two strip categories on
$\Theta$ are derived equivalent.

\begin{corollary}
\label{cor:rectangular-lattice}
For $m,n\geq1$,
\[
 \Db{\kk J_{m,n}}\simeq\Db{A_{n+1}^{m}}
 \simeq\Db{A_{m+1}^{n}}.
\]
These equivalences recover \cite[Theorem~E]{Gottesman}, after
transposing the rectangle and converting the indices as in
Remark~\ref{rem:index-conversion}.
\end{corollary}

\begin{proof}
Take $\mu=\varnothing$ and $\lambda=(n^m)$ in
Theorem~\ref{thm:skew-strip-tilting}. The selected labels exhaust
$\os_m(m+n)$, so the corner is the whole algebra.
\end{proof}

\begin{remark}
Gottesman's construction uses shifted antichain interval modules
and an explicit presentation of their morphisms. The relative
Kan-extension construction recovers the rectangular equivalence
one coordinate at a time. Because the construction is defined
on each coordinate interval, the same proof gives the skew-diagram
and transposition statements.
\end{remark}

The interval hypothesis is essential, as the following example shows.
\begin{proposition}\label{prop:noninterval-counterexample}
Let $\Xi=\{12,13,14,15,23,24,34\}\subset\os_2(5)$.
This is a coordinate lower ideal, but its incidence algebra is not
derived equivalent to the higher Auslander corner on the same labels.
\end{proposition}

\begin{proof}
In the displayed vertex order, the Cartan matrices of the incidence
algebra and the corner, with rows indexed by targets and columns by
sources, are respectively
\[
C_0=\begin{pmatrix}
1&0&0&0&0&0&0\\
1&1&0&0&0&0&0\\
1&1&1&0&0&0&0\\
1&1&1&1&0&0&0\\
1&1&0&0&1&0&0\\
1&1&1&0&1&1&0\\
1&1&1&0&1&1&1
\end{pmatrix},\qquad
C_1=\begin{pmatrix}
1&0&0&0&0&0&0\\
1&1&0&0&0&0&0\\
1&1&1&0&0&0&0\\
1&1&1&1&0&0&0\\
0&1&0&0&1&0&0\\
0&1&1&0&1&1&0\\
0&0&1&0&0&1&1
\end{pmatrix}.
\]
Both algebras are directed and have finite global dimension. Their
inverse Coxeter matrices $\Phi_i=-C_i^{-T}C_i$ satisfy
$\operatorname{tr}(\Phi_0^3)=2$ and
$\operatorname{tr}(\Phi_1^3)=-1$; equivalently, their characteristic
polynomials are
\[
 t^7+t^6-t^4-t^3+t+1,\qquad t^7+t^6+t+1.
\]
Derived equivalence would conjugate the Coxeter transformations on
$K_0$, contradicting these traces.
\end{proof}

\section{Dyck diagrams, intrinsic corners, and transposition}
\label{sec:dyck}

\subsection{Coordinates for paths}

Throughout this section, fix integers $a,b\geq2$. If $a=1$ or $b=1$,
the rational Dyck poset is a singleton and will always be treated
directly, without dimension-zero higher Auslander notation.

Let $\rev{L_{a,b}}$ denote the set of all north-east lattice paths from
$(0,0)$ to $(a,b)$ in the $a\times b$ rectangle. Every such path has
$a$ horizontal steps and $b$ vertical steps. Label its $(a+b)$ steps
consecutively by $1,2,\ldots,a+b$,
and write $c(\ell)=(x_1,\ldots,x_a)$
for the increasing sequence of the labels of its horizontal steps.
Then $1\leq x_1<\cdots<x_a\leq a+b$.
Conversely, every strictly increasing $a$-tuple satisfying these
inequalities determines a unique path in $\rev{L_{a,b}}$.

Immediately before the $i$-th horizontal step, the path has taken
$i-1$ horizontal steps and $x_i-i$ vertical steps. Hence it is at the
point $(i-1,x_i-i)$.
Immediately after the same horizontal step, it is at $(i,x_i-i)$.

\subsection{The above- and below-diagonal Dyck posets}

Let $\Delta_{a,b}$ be the diagonal segment joining $(0,0)$ to $(a,b)$.
Its equation is $q=\frac ba p$.
For a point $(p,q)$ in the rectangle, define its signed displacement
from the diagonal by $\delta_{a,b}(p,q)=q-\frac ba p$.
Thus $(p,q)$ lies weakly above the diagonal if
$\delta_{a,b}(p,q)\geq0$, and it lies weakly below the diagonal if
$\delta_{a,b}(p,q)\leq0$.

We define
\[
\Dyck^{\mathrm{above}}_{a,b}
=
\left\{
\ell\in L_{a,b}:
\delta_{a,b}(p,q)\geq0
\text{ for every point }(p,q)\text{ on }\ell
\right\}
\]
and
\[
\Dyck^{\mathrm{below}}_{a,b}
=
\left\{
\ell\in L_{a,b}:
\delta_{a,b}(p,q)\leq0
\text{ for every point }(p,q)\text{ on }\ell
\right\}.
\]

The superscript specifies the side of the diagonal, and $a,b$
are the horizontal and vertical side lengths. Weak inequalities
allow paths to meet the diagonal at intermediate vertices.
We equip both sets with the coordinatewise order induced by the
horizontal-step coordinates. Namely, if
$c(\ell)=(x_1,\ldots,x_a),
c(\ell')=(x'_1,\ldots,x'_a)$,
then
$\ell\leq\ell'
\quad\Longleftrightarrow\quad
x_i\leq x'_i$ for every $1\leq i\leq a$.
With this order,
$\Dyck^{\mathrm{above}}_{a,b}$ and
$\Dyck^{\mathrm{below}}_{a,b}$ are finite posets.

We next translate the two geometric conditions into inequalities for
the coordinates $x_i$.

A path stays weakly below the diagonal precisely when the point
immediately before each horizontal step does so. Indeed, the signed
displacement increases along vertical steps and decreases along
horizontal steps. Thus the condition is
$x_i-i
\leq
\frac{b(i-1)}a$.
Since $x_i-i$ is an integer, this is equivalent to
\begin{equation}
x_i
\leq
i+\left\lfloor\frac{b(i-1)}a\right\rfloor
\qquad
(1\leq i\leq a).
\label{eq:dyck-bound}
\end{equation}
Thus, through the coordinate encoding $c$, we identify
\[
\Dyck^{\mathrm{below}}_{a,b}
=
\left\{
(x_1,\ldots,x_a):
\begin{array}{l}
1\leq x_1<\cdots<x_a\leq a+b,\\[2pt]
x_i\leq
i+\left\lfloor\dfrac{b(i-1)}a\right\rfloor
\text{ for every }i
\end{array}
\right\}.
\]

Similarly, a path stays weakly above the diagonal precisely when
the point immediately after each horizontal step does so. This gives
$x_i-i
\geq
\frac{bi}{a}$.
Since $x_i-i$ is an integer, this is equivalent to $x_i
\geq
i+\left\lceil\frac{bi}{a}\right\rceil$ for $1\leq i\leq a$.
Thus
$\Dyck^{\mathrm{above}}_{a,b}
=
\left\{
(x_1,\ldots,x_a):
\begin{array}{l}
1\leq x_1<\cdots<x_a\leq a+b,\\[2pt]
x_i\geq
i+\left\lceil\dfrac{bi}{a}\right\rceil
\text{ for every }i
\end{array}
\right\}$.
We use the below-diagonal convention
$\Dyck^{\mathrm{below}}_{a,b}$, whereas
Chapoton-Ladkani-Rognerud use the above-diagonal convention
$\Dyck^{\mathrm{above}}_{a,b}$.

\subsection{Deleting the forced prefix}

Set $s=\left\lceil\frac ab\right\rceil,m=a-s$.

\begin{lemma}\label{lem:forced-prefix}
For a below-diagonal rational Dyck path,
\[
 x_i=i\qquad(1\leq i\leq s).
\]
\end{lemma}

\begin{proof}
If $1\leq i\leq s$, then $i-1\leq s-1<\frac ab$,
so $b(i-1)<a$ and $\left\lfloor\frac{b(i-1)}a\right\rfloor=0$.
Equation \eqref{eq:dyck-bound} gives $x_i\leq i$.  Strict increase and
positivity give $x_i\geq i$, hence equality.
\end{proof}

Define
$\beta(x_1,\ldots,x_a)
 =(x_{s+1}-s,\ldots,x_a-s)$.
For $1\leq j\leq m$, put $H_j^{a,b}
 =
 j+\left\lfloor\frac{b(s+j-1)}a\right\rfloor$.

\begin{proposition}
\label{prop:dyck-staircase}
For $a,b\geq2$, the map $\beta$ is an order isomorphism from the poset of below-diagonal
rational $(a,b)$-Dyck paths to the staircase
\begin{equation*}
 \Omega_{a,b}
 =
 \{1\leq z_1<\cdots<z_m:
 z_j\leq H_j^{a,b}\}.
\end{equation*}
\end{proposition}

\begin{proof}
By Lemma \ref{lem:forced-prefix}, the first $s$ coordinates are fixed.  For
$z_j=x_{s+j}-s$, equation \eqref{eq:dyck-bound} becomes $z_j
 \leq j+\left\lfloor\frac{b(s+j-1)}a\right\rfloor
 =H_j^{a,b}$.
Conversely, any strictly increasing tuple satisfying these bounds reconstructs
a unique Dyck coordinate tuple by adjoining $1,2,\ldots,s$ and adding $s$
to the remaining coordinates.  The boundary inequality
$s<x_{s+1}=z_1+s$
follows from $z_1\geq1$.  Since both path orders are coordinatewise, the
bijection preserves and reflects order.
\end{proof}

\subsection{Identification of the corner}

Let $A'=A_{b+1}^{\,a-s}$ and let $e=\sum_{\ell\in\Dyck^{\mathrm{below}}_{a,b}}e_{\beta(\ell)}$.
For $a,b\geq2$, define the staircase corner $B^{\mathrm{st}}_{a,b}=eA'e$.
When $\gcd(a,b)=1$, the Dyck corner $B_0$ of
\cite[Proposition~4.33]{Xing} is the endomorphism algebra of
the Dyck-indexed projective summand. The coordinate identification
in the following theorem gives
 $B_0\cong eA'e=B^{\mathrm{st}}_{a,b}$.

\begin{theorem}
\label{prop:B0-C}
\label{thm:intrinsic-dyck-corner}
Let $a,b\geq2$, and put
$s=\left\lceil\frac ab\right\rceil,
 t=\left\lceil\frac ba\right\rceil,
 m=a-s,\quad n=b-t$.
The algebra $B^{\mathrm{st}}_{a,b}$ has the following descriptions.
Its intrinsic category is the horizontal-strip category
$\mathcal H_{\lambda_{a,b}/\varnothing}$ of
Theorem~\ref{thm:skew-strip-tilting}.
\begin{enumerate}
\item Its objects are the below-diagonal Dyck paths, with horizontal-step
coordinates $x=(x_1,\ldots,x_a)$. There is one basis morphism
$u_{yx}:x\to y$ precisely when
\[
 x_i\leq y_i\quad(1\leq i\leq a),\qquad
 y_i<x_{i+1}\quad(1\leq i<a).
\]
For composable basis morphisms, $u_{zy}u_{yx}=u_{zx}$ if the outer
pair satisfies these inequalities, and the product is zero otherwise.
This specifies the algebra without choosing an ambient Auslander algebra.
\item Under $\beta$, this presentation is
$\operatorname{Alg}(\C_{m-1}(\Omega_{a,b}))$. Moreover, one can
use the smaller ambient algebra
\[
 B^{\mathrm{st}}_{a,b}\cong
 e_{\Omega_{a,b}}A_{n+1}^{m}e_{\Omega_{a,b}},
 \qquad \Omega_{a,b}\subseteq\os_m(m+n).
\]
Here $m+n$ is the least ambient size supporting these fixed
compressed coordinate labels.
\item Let
$\lambda_{a,b}
 =\left(\left\lfloor\frac{b(a-1)}a\right\rfloor,\ldots,
         \left\lfloor\frac{bs}a\right\rfloor\right)$.
Then $\Dyck^{\mathrm{below}}_{a,b}\cong J(\lambda_{a,b})$,
$\lambda_{b,a}=\lambda_{a,b}^{\,T}$, and
$\Db{B^{\mathrm{st}}_{a,b}}\simeq
 \Db{B^{\mathrm{st}}_{b,a}}$.
These statements require no coprimality hypothesis. For coprime
parameters the algebra is the Dyck corner $B_0$.
\end{enumerate}
\end{theorem}

\begin{proof}
The first $s$ coordinates of every Dyck path are $1,\ldots,s$.
Thus all interlacing inequalities inside this prefix, as well as
the inequality across its last coordinate, hold automatically.
Deleting it and subtracting $s$ preserves every remaining
inequality and the outer-pair multiplication. This proves (1)
and the first assertion of (2), directly from the standard higher
Auslander Hom formula. In particular, the same presentation is the
endomorphism algebra of the Dyck-indexed right projective summand
of $A_{b+1}^{a}$; the identification does not use a replication theorem.

The largest compressed coordinate is
$H_m^{a,b}
 =m+\left\lfloor\frac{b(a-1)}a\right\rfloor=m+n$.
The Hom and composition rules on the selected labels are independent
of larger ambient bounds, proving the smaller-corner assertion.
The effective upper-bound tuple itself is allowed, so this largest
coordinate actually occurs. This proves the stated minimality for
the fixed labels.

For a Dyck path put $c_i=x_i-i$. The numbers $c_i$ increase weakly,
the first $s$ are zero, and $c_i\leq\lfloor b(i-1)/a\rfloor$.
Reversing the remaining sequence gives the row lengths of a
partition contained in $\lambda_{a,b}$. This is an order-preserving
bijection with the order ideals of that diagram. For $1\leq q\leq n$,
the $q$th row of the transposed diagram has length
\[
 \#\{k:s\leq k\leq a-1,\ \lfloor bk/a\rfloor\geq q\}
 =a-\left\lceil\frac{aq}b\right\rceil
 =\left\lfloor\frac{a(b-q)}b\right\rfloor.
\]
These are exactly the rows of $\lambda_{b,a}$. Transposition is an
isomorphism of the cell posets, hence of their ideal lattices.
The transposition equivalence of
Theorem~\ref{thm:skew-strip-tilting} proves (3). Finally, the definition of the Dyck corner $B_0$ is the
Dyck-indexed projective endomorphism algebra, giving the asserted
attribution in the coprime case.
\end{proof}

\begin{theorem}
For integers $a,b\geq2$, we obtain
\begin{equation}
 \Db{\kk\Dyck^{\mathrm{below}}_{a,b}}\simeq\Db{B^{\mathrm{st}}_{a,b}}.
\label{eq:below-B0}
\end{equation}
When $\gcd(a,b)=1$, the algebra on the right is the Dyck corner $B_0$.
\end{theorem}

\begin{proof}
Combine Proposition \ref{prop:dyck-staircase}, Theorems \ref{thm:staircase} and \ref{prop:B0-C}.
No coprimality hypothesis is needed for this equivalence.
\end{proof}

\subsection{Comparison of the above and below conventions}

To compare the two conventions, put $r=a+b$ and define
\begin{equation}
 \rho(x_1,\ldots,x_a)=(r+1-x_a,\ldots,r+1-x_1).
\label{eq:rho}
\end{equation}
This reverses the step word. Geometrically, it rotates the rectangle
by $180^\circ$ and reverses the direction of traversal.

\begin{lemma}
\label{lem:above-below}
The map
\[
\rho:
\Dyck^{\mathrm{above}}_{a,b}
\longrightarrow
\Dyck^{\mathrm{below}}_{a,b}
\]
is an order anti-isomorphism. Equivalently, it is an order isomorphism
\[
\Dyck^{\mathrm{above}}_{a,b}
\xrightarrow{\sim}
\left(\Dyck^{\mathrm{below}}_{a,b}\right)^{\op}.
\]
Consequently, $\kk\Dyck^{\mathrm{above}}_{a,b}
\cong(\kk\Dyck^{\mathrm{below}}_{a,b})^{\op}$.
\end{lemma}

\begin{proof}
The half-turn sends $(p,q)$ to $(a-p,b-q)$, and
$\delta_{a,b}(a-p,b-q)=-\delta_{a,b}(p,q)$.
It therefore exchanges the two path sets. Formula~\eqref{eq:rho}
is an involution, so it gives a bijection. Moreover,
\[
 x_i\leq y_i\text{ for all }i
 \quad\Longleftrightarrow\quad
 r+1-x_{a+1-i}\geq r+1-y_{a+1-i}\text{ for all }i.
\]
Hence the bijection reverses the coordinatewise order. The induced
incidence-algebra isomorphism follows.
\end{proof}

\begin{corollary}
\label{cor:above-B0}
For integers $a,b\geq2$,
\[
 \Db{\kk\Dyck^{\mathrm{above}}_{a,b}}
 \simeq\Db{(B^{\mathrm{st}}_{a,b})^{\op}}.
\]
When $\gcd(a,b)=1$, the right-hand side is $\Db{B_0^{\op}}$.
\end{corollary}

\begin{proof}
 Apply Lemma \ref{lem:above-below} and take opposites in \eqref{eq:below-B0}. A derived equivalence between finite-dimensional
algebras induces a derived equivalence between their opposite algebras.
\end{proof}

\section{Serre windows, direct tensor products, and the CLR corollary}
\label{sec:conjecture}

\subsection{Tensoring a tilting equivalence}

The following standard tensor property is the case of
\cite[Theorem~2.1]{Rickard1991} with one regular factor.

\begin{lemma}
\label{lem:tensor-tilting}
Let $A,B,C$ be finite-dimensional $\kk$-algebras and let
$T\in\Kb{\pmodcat{A}}$ be tilting with
$\End(T)^{\op}\cong B$. Then $C\otimes_\kk T$ is tilting over
$C\otimes_\kk A$, with opposite endomorphism algebra
$C\otimes_\kk B$.
\end{lemma}
\begin{proof}
For left modules, the natural dg-algebra identification
\[
 \Hom^\bullet_{C\otimes A}(C\otimes T,C\otimes T)
 \cong C^{\op}\otimes_\kk\Hom_A^\bullet(T,T)
\]
gives rigidity and the stated multiplication in degree zero.
Perfectness is immediate. Applying $C\otimes_\kk-$ to
$A\in\thick(T)$ gives $C\otimes A\in\thick(C\otimes T)$,
which proves generation. Rickard's theorem gives the equivalence.
\end{proof}

\subsection{Serre untwisting of a finite window}

The next construction untwists the Serre action on each layer and
identifies the resulting generators with the simples of a chain
tensor product. Unit extension classes then recover its projectives.
We use the standard minimal-model and twisted-complex formalism
of \cite[Sections~3.3 and~7]{KellerAInfinity}.

\begin{theorem}
\label{thm:direct-Serre-window}
Let $A$ be a finite-dimensional algebra of finite global dimension,
$S_A$ its Serre functor on $\Db{A}$, and $P\in\perf(A)$.
Put $B=\End_{\Db{A}}(P)^{\op}$ and $R=\kk\ovr A_r$, with $r\geq1$.
Assume that $B$ has finite global dimension,
$\bigoplus_{i=0}^{r-1}S_A^iP$ is tilting, and
\[
 \Hom(P,S_A^qP)=0
 \quad(-r<q<r,\ q\notin\{0,1\}).
\]
Let $j:\Db{B}\to\Db{A}$ be the standard fully faithful functor
with image $\thick(P)$ and $j(B)=P$, and let $S_B$ be its
source category's Serre functor. Define
\begin{equation}
 E_i=S_A^i j(S_B^{-i}B)[-i],\qquad 0\leq i<r.
\label{eq:untwisted-window}
\end{equation}
There is an enhanced equivalence
$\Db{R\otimes_\kk B}\xrightarrow{\sim}\Db{A}$
sending $U_{i+1}\otimes_\kk B$ to $E_i$, where $U_t$ is the
simple left $R$-module at $t$.
In particular, the images $V_i$ of the projectives
$Re_{i+1}\otimes_\kk B$ give an explicit tilting object
$V=\bigoplus_{i=0}^{r-1}V_i$ with
$\End(V)^{\op}\cong R\otimes_\kk B$.
They can be constructed recursively by
\begin{equation}
 V_{r-1}=E_{r-1},\qquad
 V_{i+1}\longrightarrow V_i\longrightarrow E_i
   \xrightarrow{\delta_i}V_{i+1}[1],
\label{eq:untwisted-projective-cones}
\end{equation}
Here $\delta_i$ is the extension class corresponding to
$1\in B^{\op}$ under the natural identifications
$\Hom(E_i,V_{i+1}[1])\cong\Hom(E_i,E_{i+1}[1])\cong B^{\op}$.
\end{theorem}

\begin{proof}
Rigidity of the given window makes $P$ a tilting generator for
$\thick(P)$, so the standard functor $j$ exists and is fully faithful.
For $j_i=S_A^i j$ and $0\leq i,t<r$, the hypotheses and
rigidity imply that
\[
 \Hom(j_iX,j_tY[p])=0
 \qquad(t-i\notin\{0,1\})
\]
for all $X,Y\in\Db{B}$ and $p\in\mathbb Z$.
Indeed, the assertion first holds for $X=Y=B$ and then extends
in each variable through their thick closures. For adjacent layers,
Serre duality and full faithfulness give natural identifications
\[
 \Hom(j_iX,j_{i+1}Y[p])
 \cong D\Hom_B(Y,X[-p])
 \cong\Hom_B(X,S_BY[p]).
\]
Substituting $X=S_B^{-i}B$, $Y=S_B^{-(i+1)}B$ and the shifts
in \eqref{eq:untwisted-window} yields
\begin{equation}
 \Hom(E_i,E_t[p])\cong
 \begin{cases}
  B^{\op},&t=i,\ p=0,\\
  B^{\op},&t=i+1,\ p=1,\\
  0,&\text{otherwise}.
 \end{cases}
\label{eq:untwisted-simple-pattern}
\end{equation}
The adjacent group is the regular $B^{\op}$-bimodule: naturality
of Serre duality identifies both actions with composition on
$S_B^{-i}B$. Thus the inverse intrinsic Serre functor removes
the adjacent Serre twist, including its bimodule actions.

It remains to determine the enhancement. Take a strictly unital
minimal $A_\infty$ model on the fixed objects $E_0,\ldots,E_{r-1}$.
A composable $h$-tuple with $c$ forward edges has total input
degree $c$ and endpoints at distance $c$. For $h\geq3$, the
output degree of $m_h$ is $c+2-h$. If $c\geq2$, the target Hom
space is zero; if $c=0$ or $1$, its only possible degree is $c$,
whereas $c+2-h<c$. Thus every higher operation vanishes.
The argument retains the object idempotents and works over any
field, including when $B$ is not semisimple.

The collection $(U_{i+1}\otimes B)_{0\leq i<r}$ has the same
minimal model. The resolutions
$0\to Re_{t+1}\to Re_t\to U_t\to0$ for $t<r$, and
$U_r=Re_r$, give precisely \eqref{eq:untwisted-simple-pattern}
by the Hom K\"unneth formula. Its diagonal and adjacent products
are the same regular-bimodule actions, and two adjacent products
are zero. The identical degree argument removes its higher
operations. The isomorphism of minimal models therefore extends
to twisted complexes and idempotent completions.

Each $S_B^{-i}B$ generates $\Db{B}$, so
$\thick(E_i)=S_A^i\thick(P)$. The $E_i$ together generate
$\Db{A}$ by the assumed window generation.
The $U_t\otimes B$ generate $\Db{R\otimes B}$ by the finite
layer filtration; the tensor algebra has finite global dimension
since $R$ is a split hereditary chain algebra. This proves the
enhanced equivalence.

Finally, the projective short exact sequences over $R$ transport
to \eqref{eq:untwisted-projective-cones}. Successive filtrations
of $V_{i+1}$ have factors $E_{i+1},\ldots,E_{r-1}$; by
\eqref{eq:untwisted-simple-pattern}, projection to the first factor
induces the displayed Hom isomorphism defining $\delta_i$.
Thus the cones are specified by the unit extension classes.
Their direct sum is the image of the regular tensor-algebra
generator, proving both tilting and its endomorphism multiplication.
\end{proof}

\begin{remark}
The construction expresses the Serre-duality mechanism behind
Xing's orbit blocks \cite[Section~3.2]{Xing} and Ladkani's
chain-product equivalence \cite[Corollary~1.3]{Ladkani} through
the explicit generators \eqref{eq:untwisted-window} and
\eqref{eq:untwisted-projective-cones}. Its hypotheses are
independent of Dyck combinatorics. The block-algebra method has
the same scope under these hypotheses, and Ladkani's algebraic
statement requires only a Gorenstein seed. Here finite global
dimension of $B$ supplies the inverse Serre functors on $\Db{B}$.
The contribution of this formulation is the direct construction
and enhancement check; the resulting block-to-chain equivalence
retains its original attribution.
\end{remark}

\subsection{Weighted exchange completion and interlacing separation}

We now establish the combinatorial input for the window:
weighted exchange completion and interlacing separation for ordered
labels. Generation and rigidity will follow by applying exact
sequences to this label theorem. For the Dyck Serre window of
\cite[Section~4.4]{Xing}, the height-and-defect idea of
\cite[Lemma~4.23]{Xing} is expressed here by coordinate inequalities
and a well-founded integer induction, including non-coprime parameters.
The equivalence and its
homological consequences are deferred to the next layer.

For the window argument, work with right modules and write
$\mathscr D_{\mathrm r}(A)=D^b(\operatorname{mod}\!\text{-}A)$.
Fix $d,n\geq1$, and put
\[
 r=d+n,\qquad L=r+1,\qquad k=d+1,\qquad A=A_{n+1}^{d}.
\]
The vertices of $A$ are the increasing $d$-tuples in $\{1,\ldots,r\}$.
For an increasing $k$-tuple $X=(x_1,\ldots,x_k)$ in
$\{1,\ldots,L\}$, let $M_X$ be the standard interval module
\[
 (M_X)_z=\begin{cases}
 \kk,&x_j\leq z_j<x_{j+1}\quad(1\leq j\leq d),\\
 0,&\text{otherwise},
 \end{cases}
\]
with identity structure maps on its support. In these strict-coordinate
conventions, the standard module calculations in
\cite[Theorem~3.6 and Proposition~3.12]{OppermannThomas} read
\begin{equation}
 \Hom_A(M_X,M_Y)=
 \begin{cases}
 \kk,&x_1\leq y_1<x_2\leq y_2<\cdots<x_k\leq y_k,\\
 0,&\text{otherwise}.
 \end{cases}
\label{eq:window-module-hom}
\end{equation}
The canonical nonzero map is the identity on the common support. The modules
$M_{(1,z_1+1,\ldots,z_d+1)}$ and $M_{(z_1,\ldots,z_d,L)}$
are respectively $P_z=e_zA$ and $I_z=D(Ae_z)$.

Regard $X$ as the horizontal-step positions of a word with $d+1$
letters $E$ and $n$ letters $N$. Let $\operatorname{rot}$ move its first
letter to the end. For a Dyck word $\ell$ with $d$ letters $E$ and $n$
letters $N$, write $\bar\ell=E\ell$, so $P_\ell=M_{\bar\ell}$.

The following formulas are standard interval-module calculations;
the rotation is also recorded in \cite[Proposition~4.13]{Xing}.
We give their coordinate proof to fix the maps and shifts.

\begin{lemma}
If $U=(u_1<\cdots<u_{k+1})\subseteq\{1,\ldots,L\}$, there is an exact
sequence
\begin{equation}
 0\longrightarrow M_{U\setminus\{u_{k+1}\}}
 \longrightarrow M_{U\setminus\{u_k\}}
 \longrightarrow\cdots\longrightarrow
 M_{U\setminus\{u_1\}}\longrightarrow0.
\label{eq:window-deletion-sequence}
\end{equation}
For the Serre functor $S$ of $\mathscr D_{\mathrm r}(A)$,
\begin{equation}
 S M_X\cong
 \begin{cases}
 M_{\operatorname{rot}(X)},&x_1=1,\\
 M_{\operatorname{rot}(X)}[d],&x_1>1.
 \end{cases}
\label{eq:window-serre-rotation}
\end{equation}
Consequently, for $0\leq q\leq L$,
$S^qP_\ell\cong M_{\operatorname{rot}^q(\bar\ell)}[dv_q]$,
where $v_q$ counts the letters $N$ among the first $q$ letters of
$\bar\ell$.
\end{lemma}

\begin{proof}
Take the consecutive nonzero maps from \eqref{eq:window-module-hom} in
\eqref{eq:window-deletion-sequence}. At a vertex $z$, partition
$\{1,\ldots,L\}$ into the $d+1$ intervals cut out by $z_1,\ldots,z_d$:
\[
 [1,z_1],\ (z_1,z_2],\ \ldots,\ (z_{d-1},z_d],\ (z_d,L].
\]
A term of the sequence is nonzero at $z$ precisely when its label has
one entry in each interval. Thus either every term is zero, or exactly
one interval contains two entries of $U$ and every other interval contains
one. In the latter case, precisely two consecutive terms are nonzero,
and the map between them is the identity. This proves exactness
pointwise, including the zero composites.

If $x_1=1$, the formula follows from $S(P_z)=I_z$.
If $x_1>1$, apply \eqref{eq:window-deletion-sequence} to $\{1\}\cup X$.
All terms except $M_X$ contain $1$, so this is a projective resolution
of length $d$. Applying $S$ replaces these projectives by injectives.
The resulting complex is the sequence associated to
$\{x_1-1,\ldots,x_k-1,L\}$ with its initial term
$M_{(x_1-1,\ldots,x_k-1)}$ removed. It has this module as its sole
cohomology, in degree $-d$. This proves
\eqref{eq:window-serre-rotation}; iteration gives the last assertion.
\end{proof}

For $X=(x_1<\cdots<x_k)$, define
\begin{equation}
 \begin{split}
 a_j(X)&=d x_j-rj,\\
 j(X)&=\max\{j:a_j(X)=\max_t a_t(X)\},\\
 h(X)&=x_{j(X)}-j(X),\\
 \mu(X)&=\sum_{t>j(X)}
       \max\{0,a_t(X)-a_{j(X)}(X)+n\}^{2}.
 \end{split}
\label{eq:window-integer-measure}
\end{equation}
Thus $0\leq h(X)\leq n$ and $\mu(X)\in\mathbb Z_{\geq0}$.
Choosing the last maximizing index fixes the cut even when several
indices attain the maximum.
Call $X$ a seed if $h(X)>0$ and $\mu(X)=0$.

\begin{theorem}
\label{thm:boundary-exchange}

For all positive $d,n$, the following statements hold.
\begin{enumerate}
\item The seeds are exactly the labels
$\operatorname{rot}^q(\bar\ell)$, with
$\ell\in\Dyck^{\mathrm{below}}_{d,n}$ and $1\leq q\leq r$.
\item Suppose $h(X)>0$ and $\mu(X)>0$. Write $j=j(X)$ and choose
$l>j$ such that $a_l(X)>a_j(X)-n$. Then $t=x_l-1\notin X$.
For $1\leq i\leq k$, set
$Y_i=(X\cup\{t\})\setminus\{x_i\}$, in increasing order. Then
\[
 \begin{cases}
 h(Y_i)>h(X),&i<l,\\
 h(Y_i)=h(X),\quad\mu(Y_i)<\mu(X),&i\geq l.
 \end{cases}
\]
\end{enumerate}
\end{theorem}

\begin{proof}
Give $E$ weight $-n$ and $N$ weight $d$. The weight of a prefix ending
at $(x,y)$ is $dy-nx$. Immediately before the $j$th horizontal step of
the word for $X$, it is $a_j(X)+n$. The maximum prefix weight occurs
at one of these positions: the prefix before the first horizontal step
has nonnegative weight, whereas the final weight is $-n$.

Let $j=j(X)$ and $M=a_j(X)+n$. The equality $\mu(X)=0$ says that
every prefix after the $j$th horizontal step has weight at most $M-n$.
Prefixes before that step have weight at most $M$. Cutting just before
that step and rotating therefore produces $E\ell$, where all prefix
weights of $\ell$ are nonpositive and its total weight is zero.
Thus $\ell$ is a Dyck word. Since $h(X)>0$, the cut position
$x_j-1$ lies between $1$ and $L-1$, giving a rotation exponent
between $1$ and $r$.

Conversely, rotate $E\ell$ by $q$, with $1\leq q\leq r$, and mark
the added letter $E$. In the rotated word, every prefix before the
marked letter has weight at most the weight $M$ immediately preceding
it, and every later prefix has weight at most $M-n$. Hence the marked
letter is the last maximizing horizontal step and $\mu=0$. Its vertical
coordinate is positive because the suffix moved before it contains the
last letter of $\ell$, which is $N$. Thus $h>0$. This proves (1),
without a coprimality assumption.

For (2), abbreviate $a_i=a_i(X)$, $h=h(X)$ and $a=a_j$.
If $x_l=x_{l-1}+1$, then $a_l=a_{l-1}-n\leq a-n$, a contradiction.
Thus $t\notin X$. Also, putting $h_l=x_l-l$, the strict inequality
$a_l>a_j-n$ gives
\begin{equation}
 d(h_l-h)>n(l-j-1)\geq0,
 \qquad\text{so }h_l>h.
\label{eq:window-defect-height}
\end{equation}

If $i<l$, the new $a$-values at positions $i,\ldots,l-2$ are
$a_{i+1}+r,\ldots,a_{l-1}+r$, and the value at $l-1$ is $a_l+n>a$;
all remaining values are unchanged. If $i<j$, the last maximum is
$a_j+r$, now at position $j-1$, so the new height is $h+1$.
If $j\leq i<l$, any new maximum lies in the changed positions.
A shifted old entry has height $x_u-(u-1)\geq h+1$, since $u>j$.
The inserted entry has height $h_l>h$ by
\eqref{eq:window-defect-height}. Thus $h(Y_i)>h$ in either case.

If $i\geq l$, the new values at positions $l,\ldots,i$ are
\[
 a_l-d,\ a_l-r,\ a_{l+1}-r,\ldots,a_{i-1}-r,
\]
with only the first entry present when $i=l$. The last maximizing
index remains $j$. Every entry decreased by $r=d+n$ contributes zero
to the new defect, and the contribution from $a_l-d$ is strictly
smaller than that from $a_l$, which was positive. The other contributions
are unchanged or removed. Hence $h(Y_i)=h$ and
$\mu(Y_i)<\mu(X)$.
\end{proof}

The preceding local exchange has a global formulation that does
not mention modules, Serre functors, or derived equivalences.
Let $\mathscr X_{d,n}$ be the increasing $(d+1)$-tuples in
$\{1,\ldots,d+n+1\}$. Let $\mathscr P_{d,n}$ be the labels
of the words $E\ell$, where $\ell$ is a below-diagonal Dyck word,
and put $\mathscr S_{d,n}
 =\bigcup_{q=1}^{d+n}\operatorname{rot}^q\mathscr P_{d,n}$.
For $\mathscr S\subseteq\mathscr X_{d,n}$, define its
\emph{deletion-exchange closure} $\operatorname{cl}_{\partial}(\mathscr S)$
to be the smallest subset containing $\mathscr S$ with the following
property: if $U$ is any $(d+2)$-subset and all but one of its
$(d+1)$-element deletion faces belong to the subset, then the
remaining face belongs to it as well. This is a closure of finite
label sets. Its categorical use will be justified by the exact
deletion sequences, rather than by a rank count.

\begin{theorem}
\label{thm:weighted-exchange-completion}
For all $d,n\geq1$, the following hold.
\begin{enumerate}
\item Every label of positive height is recovered from the rotated
Dyck seeds:
\begin{equation}
 \{X\in\mathscr X_{d,n}:h(X)>0\}
 \subseteq\operatorname{cl}_{\partial}(\mathscr S_{d,n}).
\label{eq:weighted-exchange-closure}
\end{equation}
In particular every label with $x_{d+1}=d+n+1$ lies in this closure.
Each recovery can be chosen so that all the other faces have strictly
smaller pair $(n-h,\mu)$ in lexicographic order.
\item There are no interlacing pairs
\[
 X\preceq Y,\qquad
 X\in\mathscr P_{d,n},\quad
 Y\in\operatorname{rot}^q\mathscr P_{d,n},\quad 2\leq q\leq d+n,
\]
if and only if $\gcd(d,n)=1$.
For $g=\gcd(d,n)>1$, the repeated word
$\ell=(E^{d/g}N^{n/g})^g$ supplies an explicit pair with
$X=E\ell$ and $Y=\operatorname{rot}^{d/g+n/g+1}(X)$.
\item Put $r=d+n$. The number of Dyck labels satisfies
\begin{equation*}
 r\,|\Dyck^{\mathrm{below}}_{d,n}|\geq\binom rd,
 \qquad\text{with equality if and only if }\gcd(d,n)=1.
\end{equation*}
\end{enumerate}
\end{theorem}

\begin{proof}
By Theorem~\ref{thm:boundary-exchange}(1), the seed labels are
exactly $\mathscr S_{d,n}$. Induct on the pair of nonnegative
integers $(n-h(X),\mu(X))$. If $\mu(X)=0$, the label is a seed.
Otherwise, Theorem~\ref{thm:boundary-exchange}(2) produces
$U=X\cup\{t\}$ whose other deletion faces have positive height
and strictly smaller pair. They are in the closure by induction,
so its defining exchange puts $X$ there. This proves
\eqref{eq:weighted-exchange-closure} and gives a terminating
recovery algorithm for every parameter pair.

If $x_{d+1}=d+n+1$, the last horizontal step is preceded by
$(d,n)$, of weight zero. A zero maximum prefix weight therefore
has last maximizing height $n$; a positive maximum has positive
vertical coordinate. In either case $h(X)>0$, proving the
assertion about these boundary labels.

Assume $\gcd(d,n)=1$. Write $X=E\ell$ and
$Y=\operatorname{rot}^q(E\ell')$, and mark the added letter in $Y$.
The seed calculation identifies it as the last maximizing horizontal
step. Its index $j$ is at most $d$: the first $q\geq2$ letters of
$E\ell'$ contain two horizontal steps. Its height $h$ is positive.
Its prefix weight $M=dh-n(j-1)$ is nonnegative, and equality would
force $d\mid j-1$, hence $j=1$ and $h=0$. Thus $M>0$.
If $X\preceq Y$, interlacing gives $y_j\leq x_{j+1}-1$.
The prefix weight just before the $(j+1)$st horizontal step of
$X=E\ell$ is at most $-n$, since $\ell$ is Dyck. Hence
\[
 M=d(y_j-j)-n(j-1)
 \leq d(x_{j+1}-(j+1))-nj+n\leq0,
\]
a contradiction. This proves separation.

For the converse, put $\alpha=d/g$, $\beta=n/g$ and
$q=\alpha+\beta+1\leq d+n-1$. The word
$\ell=(E^\alpha N^\beta)^g$ is Dyck, and
$Y=\operatorname{rot}^q(E\ell)
  =(E^\alpha N^\beta)^{g-1}E(E^\alpha N^\beta)$.
With $c=(g-1)\alpha+1\leq d$, its horizontal coordinates satisfy
\[
 y_i=\begin{cases}
 x_{i+1}-1,&1\leq i\leq c,\\
 x_i,&c<i\leq d+1.
 \end{cases}
\]
These inequalities give $X\preceq Y$, completing (2).

For (3), give $E,N$ weights $-n,d$ in words of length $r$.
Cutting after a maximal partial sum produces a Dyck word, so every
cyclic orbit contains at least one such word. If $\gcd(d,n)=1$,
the $r$ partial sums before the letters are distinct: equal sums
would give a nonempty proper subword of weight zero, forcing
its horizontal and vertical counts to be a positive integral
multiple of $(d,n)$. The same arithmetic
excludes a shorter period. Each orbit therefore has size $r$ and
has a unique maximal cut, giving $\binom rd/r$ Dyck words.
If $g=\gcd(d,n)>1$, the repeated word
$(E^{d/g}N^{n/g})^g$ has an orbit of size less than $r$.
All orbits have size at most $r$ and each contributes a Dyck word;
hence their number is strictly greater than $\binom rd/r$.
This proves the strict inequality and completes (3).
\end{proof}

Coprimality excludes nonempty proper blocks of weight zero. The
same arithmetic therefore gives both interlacing separation and
uniqueness of the maximal cyclic cut. For non-coprime parameters,
a repeated primitive block witnesses failure of both properties.
Exchange completion requires no such arithmetic condition; it is
the combinatorial input for generation in the next corollary.

The passage from this label theorem to derived generation uses the
following exact-exchange principle. If a bounded exact complex has
all but one of its terms in a thick subcategory, then its remaining
term also lies there. To see this, start from both ends of the
complex: successive short exact sequences put the images adjacent
to the omitted term in the subcategory, and the short exact sequence
through that term finishes the argument. Thus an exact realization
of every deletion exchange carries label closure to thick closure.
This is the usual homological d\'evissage principle. Here it applies
because Theorem~\ref{thm:weighted-exchange-completion} supplies an
explicit decreasing order for every exchange.

\begin{corollary}
\label{cor:coordinate-window}
For $d,n\geq1$, put $P=\bigoplus_{\ell\in\Dyck^{\mathrm{below}}_{d,n}}P_\ell, T=\bigoplus_{q=1}^{r}S^qP, \Lambda=\End_A(P)$.
Then $\thick(T)=\mathscr D_{\mathrm r}(A)$ for every $d,n$, and
\[
 T\text{ is tilting}\quad\Longleftrightarrow\quad\gcd(d,n)=1.
\]
\end{corollary}

\begin{proof}
By the Serre-rotation formula, each seed module is a shift of a
window summand, hence belongs to $\thick(T)$. The deletion exact
sequences and the exact-exchange principle carry
\eqref{eq:weighted-exchange-closure} to thick closure. All injectives
have labels ending in $d+n+1$, so $DA\in\thick(T)$. Finite global
dimension then gives generation of $\mathscr D_{\mathrm r}(A)$.

If $\gcd(d,n)=1$, the interlacing separation of
Theorem~\ref{thm:weighted-exchange-completion}(2), the module Hom
formula, projectivity of $P$, and Serre rotation give
\begin{equation}
 \Hom(P,S^qP[s])=0
 \qquad(2\leq q\leq r,\ s\in\mathbb Z).
\label{eq:coordinate-window-vanishing}
\end{equation}
The groups $\Hom(P,P[s])$ and $\Hom(P,SP[s])$ vanish for $s\ne0$,
since $P$ is projective and $SP$ is injective. Negative Serre
differences vanish by duality:
\[
 \Hom(P,S^{-q}P[s])
 \cong D\Hom(P,S^{q+1}P[-s])=0
 \qquad(1\leq q\leq r-1).
\]
These calculations cover every difference of Serre exponents
between window summands. Hence $T$ is rigid and, by the generation
argument above, tilting.

If $g>1$, use the explicit interlacing pair of
Theorem~\ref{thm:weighted-exchange-completion}(2), with
$\beta=n/g$ and $q=d/g+n/g+1$. The first $q$ letters of $E\ell$
contain $\beta$ vertical steps. The Hom and rotation formulas give
\begin{equation}
 \Hom(P_\ell,S^qP_\ell[-d\beta])\cong\kk.
\label{eq:noncoprime-window-nonrigidity}
\end{equation}
Applying $S$ gives a nonzero degree $-d\beta$ Hom between the
window summands $SP_\ell$ and $S^{q+1}P_\ell$. Hence $T$ is not
rigid.
\end{proof}

\subsection{The direct tensor equivalence and window homology}

The window gives two systems of generators. The untwisted
collection proves the chain-product equivalence used in CLR;
the original window gives the additional homological properties.
We first apply the direct construction, compare the shorter
replicated-algebra argument, and then recover the window homology.

\begin{corollary}
\label{cor:direct-window-tensor}
Suppose $\gcd(d,n)=1$, and retain $A,P,T,r,\Lambda$ from
Corollary~\ref{cor:coordinate-window}. With $R=\kk\ovr A_r$,
\begin{equation}
 \mathscr D_{\mathrm r}(A)\simeq\Db{R\otimes_\kk\Lambda^{\op}},
 \qquad \Db{A}\simeq\Db{R\otimes_\kk\Lambda}.
\label{eq:direct-window-equivalence}
\end{equation}
For $d,n\geq2$, $\Lambda=B^{\mathrm{st}}_{d,n}=B_0$;
if either parameter is $1$, $\Lambda=\kk$.
\end{corollary}
\begin{proof}
For the Dyck idempotent $e_D$, $P=e_DA$ and
$\Lambda=\End_A(P)=e_DAe_D$. The intrinsic corner theorem
identifies it with $B^{\mathrm{st}}_{d,n}$; its equivalence with a
finite-poset incidence algebra gives finite global dimension.
The singleton case gives $\Lambda=\kk$ directly.
Apply Theorem~\ref{thm:direct-Serre-window} in the left-module
category of $A^{\op}$, with seed $P$ and seed algebra
$B=\Lambda^{\op}$. The orbit sum $S^{-1}T$ is tilting by
Corollary~\ref{cor:coordinate-window};
\eqref{eq:coordinate-window-vanishing} and Serre duality give the
required gaps. This proves the first equivalence.
Taking opposite algebras and reversing the chain vertices gives
the second. Formulae \eqref{eq:untwisted-window} and
\eqref{eq:untwisted-projective-cones} specify its tilting complex.
\end{proof}

\begin{remark}
For the original window, put $\Gamma=\End(T)$. Serre duality
identifies its adjacent blocks with $D\Lambda$. In decreasing
Serre-exponent order, $\Gamma$ has $r$ diagonal blocks $\Lambda$,
adjacent upper blocks $D\Lambda$, and all other blocks zero.
The block actions are the natural $\Lambda$-bimodule actions;
two adjacent off-diagonal factors multiply to zero. This is
Xing's known replicated presentation
\cite[Proposition~4.25]{Xing}. Since $\Lambda$ has finite global
dimension, it is Gorenstein, and
\cite[Corollary~1.3]{Ladkani} identifies this algebra up to
derived equivalence with $\kk\ovr A_r\otimes\Lambda$.
That gives a shorter alternative tensor-product argument once
window tilting is established.

The direct proof constructs the projective generator by
\eqref{eq:untwisted-window}--\eqref{eq:untwisted-projective-cones}.
Its extra check is the degree-and-endpoint argument for formality.
Both routes use the same window generation and Serre gaps in this
application, and neither extends CLR to the uniform non-coprime
window. The broader results proved here are the relative diagram
tilt and the all-parameter exchange completion. We use the direct
route in the CLR proof and retain the replicated presentation
for comparison and attribution.
\end{remark}

The next lemma is the complete-window homology statement of
\cite[Theorem~3.6]{Xing}. We include a proof from the general
criterion in Proposition~\ref{prop:standard-window-inputs}\textup{(1)}. The period
$S^{r+1}\cong[N]$ places the new boundary term of $ST$ in shift
$[N]$, detecting the global dimension. It also identifies
$S[-N]$ with $S^{-r}$, so translates of the $r$ window exponents
cover the entire Serre orbit. This supplies the inclusion needed
for the Auslander--Reiten criterion.

\begin{lemma}
\label{lem:complete-window-homology}
Let $A$ be a finite-dimensional algebra of finite global dimension,
and let $S$ be the Serre functor on
$\mathscr D_{\mathrm r}(A)$. Suppose $N,r\geq1$,
$S^{r+1}\cong[N]$ naturally, and
$T=\bigoplus_{i=0}^{r-1}S^iX$
is a nonzero tilting complex. Set $B=\End(T)$ and choose the
right-module equivalence $F:\mathscr D_{\mathrm r}(A)\simeq
\mathscr D_{\mathrm r}(B)$ with $F(T)=B$. Then
\begin{enumerate}
\item $\gldim B=N$ and $S_B^{r+1}\cong[N]$ naturally.
\item $B\oplus DB$ is an $N\mathbb Z$-cluster tilting module.
\item With $S_N=S[-N]$, the equivalence restricts to
$F:\mathcal U_N(T)\xrightarrow{\sim}\mathcal U_N(B)$,
where $\mathcal U_N(T):=\add\{S_N^jT:j\in\mathbb Z\}
       =\add\{S^qX:q\in\mathbb Z\}$.
Both are $N\mathbb Z$-cluster tilting subcategories. Explicitly,
\[
 \mathcal U_N(B)=\add\{(B\oplus DB)[Nj]:j\in\mathbb Z\}.
\]
\end{enumerate}
\end{lemma}

\begin{proof}
The equivalence intertwines Serre functors by their defining duality
and uniqueness, so $F(ST)=DB$ and
$S_B^{r+1}\cong[N]$. Finite global dimension is preserved by
derived equivalence. Put $T'=\bigoplus_{i=1}^{r-1}S^iX$.
The natural period gives
$ST=T'\oplus S^rX=T'\oplus S^{-1}X[N]$.
For every integer $p$, rigidity of $T$ and Serre duality give
\begin{align*}
 \Hom(ST,T[p])
 &\cong\Hom(T',T[p])
       \oplus\Hom(S^{-1}X[N],T[p])\\
 &\cong\Hom(T',T[p])
       \oplus D\Hom(T,X[N-p]).
\end{align*}
Consequently, $\Ext_B^p(DB,B)$ vanishes for $p>0$ other than
$p=N$, and
\[
 \Ext_B^N(DB,B)\cong D\Hom(T,X)\ne0;
\]
the group is nonzero because $X$ is a nonzero summand of $T$.
For a module of finite projective dimension, its last minimal
projective term gives a nonzero Ext into the regular module:
the preceding differential has radical image, whereas a summand
inclusion into the regular module does not. It follows that
$\operatorname{pd}_B DB=N$. For an algebra of finite global
dimension, $\operatorname{pd}_B DB=\gldim B$. Indeed, by duality
this projective dimension is the injective dimension of the left
regular module; the same last-term argument on minimal left
projective resolutions shows that this injective dimension equals
the global dimension. This proves (1).

The period also gives $S_N\cong S^{-r}$. Therefore
\begin{equation}
 \mathcal U_N(T)
 =\add\{S^{i-rj}X:0\leq i<r,\ j\in\mathbb Z\}
 =\add\{S^qX:q\in\mathbb Z\}.
\label{eq:window-entire-orbit}
\end{equation}
In particular $ST\in\mathcal U_N(T)$, so
$DB\in\mathcal U_N(B)$. Proposition~\ref{prop:standard-window-inputs}\textup{(1)},
with the already proved bound $\gldim B=N$, implies that $B$
is $N$-representation finite and that $\mathcal U_N(B)$ is
$N$-cluster tilting. The natural isomorphism $F S_N\cong S_{B,N}F$ then gives the
restricted equivalence.

Every exponent $q$ in \eqref{eq:window-entire-orbit} is congruent
to one of $0,\ldots,r$ modulo $r+1$. Since
$S^{r+1}\cong[N]$ and the exponents $0,\ldots,r$ occur in
$T\oplus ST$, we obtain
\[
 \mathcal U_N(T)
 =\add\{(T\oplus ST)[Nj]:j\in\mathbb Z\}.
\]
Applying $F$ gives the asserted description of $\mathcal U_N(B)$.
As $N>0$, a nonzero module shifted by $Nj$ lies in the module
heart only when $j=0$. Krull--Schmidt decomposition therefore gives
\[
 \mathcal U_N(B)\cap\operatorname{mod}\!\text{-}B
 =\add(B\oplus DB).
\]
By Proposition~\ref{prop:standard-window-inputs}\textup{(1)}, this is the unique
$N$-cluster tilting subcategory of modules. Its Ext groups in
degrees $1,\ldots,N-1$ vanish by cluster tilting, and its higher
Ext groups vanish above $N$ by global dimension. Hence it is
$N\mathbb Z$-cluster tilting. The same description by shifts
proves both $[N]$-closure and vanishing in all degrees outside
$N\mathbb Z$ for the derived subcategory, completing (2) and (3).
\end{proof}

\begin{corollary}
\label{cor:window-homology}
Let $d,n\geq1$ be coprime, put $r=d+n$ and $N=dn$, and use
the Dyck window $T$ of Corollary~\ref{cor:coordinate-window}.
Then, with $B=\End_{\mathscr D_{\mathrm r}(A)}(T)$, the equivalence
$F:\mathscr D_{\mathrm r}(A_{n+1}^{d})\simeq
\mathscr D_{\mathrm r}(B)$ can be chosen so that $F(T)=B$,
and the basic projective module $P$ has
\[
 |P|=\frac{1}{d+n}\binom{d+n}{d}
\]
indecomposable summands. Furthermore,
\begin{enumerate}
\item $\gldim B=dn$ and $B\oplus DB$ is an
$dn\mathbb Z$-cluster tilting module; in particular $B$ is
$2$-subhomogeneous $dn$-representation finite.
\item $S_B^{d+n+1}\cong[dn]$ as triangle functors.
\item $F$ restricts to an equivalence of $dn\mathbb Z$-cluster
tilting subcategories
\[
 F:\add\{S_{dn}^jT:j\in\mathbb Z\}
 \xrightarrow{\sim}
 \add\{(B\oplus DB)[dnj]:j\in\mathbb Z\}.
\]
\end{enumerate}
Here $2$-subhomogeneous means precisely that the global dimension
is $dn$ and the projectives together with the injectives form the
cluster tilting subcategory. For $d,n\geq2$,
$\Lambda=B^{\mathrm{st}}_{d,n}=B_0$. These are the homological
properties of the Dyck Serre window in
\cite[Theorem~4.5]{Xing}.
\end{corollary}

\begin{proof}
Corollary~\ref{cor:coordinate-window} proves that $T$ is tilting.
Its endomorphism algebra is $B$ by definition. Apply
Lemma~\ref{lem:complete-window-homology} with $X=SP$,
$r=d+n$ and $N=dn$. Its natural-period hypothesis is
Proposition~\ref{prop:standard-window-inputs}\textup{(2)}, which concerns the full
algebra and is independent of the subsequent CLR corollary.

The projective summands of $P$ are indexed by distinct Dyck
vertices. Their number is the equality case of
Theorem~\ref{thm:weighted-exchange-completion}(3).
\end{proof}

\begin{remark}
Weighted exchanges give window generation and the coprimality
criterion; Serre untwisting gives the direct tensor equivalence.
The homological lemma uses the standard inputs of
Section~\ref{subsec:cluster-preliminaries} to recover the known
window properties. These additional conclusions form a separate
branch of the argument; the CLR proof uses the direct tensor
equivalence.
\end{remark}

\subsection{Self-oppositeness of the higher Auslander algebra}

\begin{lemma}
\label{lem:A-self-op}
The higher Auslander algebra $A_{b+1}^{a}$ is
isomorphic to its opposite.
\end{lemma}

\begin{proof}
The reflection
$\sigma(z_1,\ldots,z_a)=(a+b+1-z_a,\ldots,a+b+1-z_1)$
reverses the interlacing condition and preserves the outer-pair
criterion for nonzero multiplication. Sending each basis morphism
to the reflected reverse morphism therefore gives the required
anti-isomorphism, including all zero products.
\end{proof}

\subsection{The CLR corollary}

\begin{corollary}
\label{cor:CLR-conjecture}
For coprime positive integers $a,b$,
\[
 \Db{\kk(\ovr A_{a+b}\times\Dyck^{\mathrm{above}}_{a,b})}
 \simeq\Db{\kk L_{a,b}}.
\]
\end{corollary}
\begin{proof}
If either parameter is $1$, both path posets are chains of the
required size. Otherwise put $r=a+b$, $R=\kk\ovr A_r$, and
$A=A_{b+1}^{a}$. The diagram equivalence of
Corollary~\ref{cor:above-B0}, tensored by
Lemma~\ref{lem:tensor-tilting}, and the direct window equivalence
\eqref{eq:direct-window-equivalence} give
\begin{equation}
 \Db{R\otimes\kk\Dyck^{\mathrm{above}}_{a,b}}
 \simeq\Db{R\otimes B_0^{\op}}
 \simeq\Db{A^{\op}}\simeq\Db{A}.
\label{eq:direct-CLR-chain}
\end{equation}
The last step is Lemma~\ref{lem:A-self-op}.
The full path lattice is the ideal lattice of an $a\times b$
grid, so Corollary~\ref{cor:rectangular-lattice} gives
\begin{equation}
 \Db{\kk L_{a,b}}\simeq\Db{A}.
\label{eq:rectangle-equivalence}
\end{equation}
Combining \eqref{eq:direct-CLR-chain} and
\eqref{eq:rectangle-equivalence} proves the assertion.
Thus the proof consists of the Dyck diagram tilt, direct Serre
untwisting, and the rectangle diagram tilt. The rectangle result
recovers Gottesman's Theorem~E, with
Remark~\ref{rem:index-conversion} fixing the indices.
Neither a replicated intermediate algebra nor the window's
cluster-tilting properties are used.
\end{proof}

\begin{remark}
Coprimality enters through the rigidity calculation in
Corollary~\ref{cor:coordinate-window}. Generation holds without
it, but \eqref{eq:noncoprime-window-nonrigidity} obstructs applying
the untwisting theorem to the uniform non-coprime window.
The diagram equivalence \eqref{eq:below-B0} remains valid.
\end{remark}

\subsection{The boundary of the non-coprime formula}

\begin{proposition}
\label{prop:noncoprime-rank-obstruction}
If $a,b\geq2$ and $\gcd(a,b)>1$, then
\[
 (a+b)|\Dyck^{\mathrm{below}}_{a,b}|>\binom{a+b}{a}.
\]
Consequently, the unchanged CLR formula
$\Db{\kk(\ovr A_{a+b}\times\Dyck_{a,b})}
 \simeq\Db{\kk L_{a,b}}$
is impossible in the non-coprime case, for either Dyck convention.
\end{proposition}

\begin{proof}
The strict inequality is
Theorem~\ref{thm:weighted-exchange-completion}(3).
The Grothendieck group of a finite-poset incidence algebra is free
of rank equal to its number of vertices. Thus the two sides have
ranks $(a+b)|\Dyck^{\mathrm{below}}_{a,b}|$ and $\binom{a+b}{a}$,
which differ. Lemma~\ref{lem:above-below} gives the same obstruction
for above-diagonal paths.
\end{proof}

\begin{remark}
The non-coprime corner equivalence \eqref{eq:below-B0} remains
valid. Corollary~\ref{cor:coordinate-window} also supplies a
uniform generator, hence a dg Morita model via its full derived
endomorphism algebra \cite[Theorem~3.8(b)]{KellerDG}.
Its nonzero higher Hom groups prevent the untwisting argument for
this generator. An ordinary replacement must change the
chain--Dyck product; the rank obstruction excludes the unchanged
formula. No all-parameter orbit-weighted ordinary tensor source
is asserted here.
\end{remark}

\section{Fukaya models and geometric tilting}
\label{sec:fukaya}

\subsection{Index conventions and symmetric products}

We use the notation of Dyckerhoff--Jasso--Lekili \cite{DJL}
systematically in this section; the full Serre-period fact was
already stated in Proposition~\ref{prop:standard-window-inputs}\textup{(2)}.  Let $\mathsf A_{n,d}$
denote their higher Auslander algebra, whose vertices are the strictly
increasing $d$-tuples in $\{1,\ldots,n\}$.  Comparison with Section 
\ref{sec:preliminaries} gives the index conversion
\begin{equation}
 \mathsf A_{n,d}=A_{n-d+1}^{d}
 \quad\text{in our strict-coordinate convention}.
\label{eq:DJL-index}
\end{equation}

Let $D\subset\mathbb C$ be the closed disk and let
$\Lambda_n\subset\partial D$ be a set of $n+1$ marked points.  Put
$\Lambda_n^{(d)}
 =
 \bigcup_{p\in\Lambda_n}
 \{p\}\times\operatorname{Sym}^{d-1}(D)
 \subset\partial\operatorname{Sym}^{d}(D)$
and abbreviate
$\Wcat_n^{(d)}
 =
 \Wcat\!\left(\operatorname{Sym}^{d}(D),\Lambda_n^{(d)}\right)$.
For a finite-dimensional algebra $A$, let
$\operatorname{Perf}_{\mathrm{dg}}(A)$ denote the pretriangulated and
idempotent-complete dg category of perfect right dg $A$-modules, and put
$\perf(A)=H^0\!\left(\operatorname{Perf}_{\mathrm{dg}}(A)\right)$.
Thus $\operatorname{Perf}_{\mathrm{dg}}(A)$ is an enhancement, whereas
$\perf(A)$ is its triangulated homotopy category.  We regard dg categories
as strictly unital, uncurved $A_\infty$-categories when comparing them with
Fukaya categories, and we choose strictly unital, uncurved models throughout
this section.

We take $\Wcat_n^{(d)}$ to be the pretriangulated and idempotent-complete
partially wrapped Fukaya $A_\infty$-category of the stopped symmetric
product.  Its cohomological category $H^0(\Wcat_n^{(d)})$ is triangulated.
At the enhanced level, \cite[Theorems~1 and~2]{DJL} give
quasi-equivalences of $A_\infty$-categories
\begin{equation}
 \begin{aligned}
 \operatorname{Perf}_{\mathrm{dg}}(\mathsf A_{n,d})
   &\simeq\Wcat_n^{(d)} &&(1\leq d\leq n),\\
 \Wcat_n^{(d)}
   &\simeq\operatorname{Perf}_{\mathrm{dg}}(\mathsf A_{n,n-d})
       &&(1\leq d<n).
 \end{aligned}
\label{eq:DJL-equivalences}
\end{equation}
For $d=n$, the first equivalence identifies both categories with
$\operatorname{Perf}_{\mathrm{dg}}(\kk)$. Taking $H^0$ gives the
corresponding triangle equivalences in the stated ranges.
The first equivalence sends the indecomposable projectives to explicit
product Lagrangians $L_I$, indexed by
$I\in\os_d(n)$. 
The complementary-degree equivalence realizes the
Koszul-duality symmetry $d\leftrightarrow n-d$.

\subsection{Coordinate intervals as Lagrangian-generated subcategories}

The coordinate-interval equivalence has a direct symplectic
interpretation. The reflection below reconciles the left-module
convention of Section~\ref{sec:preliminaries} with the right-module
convention of the ambient Fukaya model.

\begin{lemma}
\label{lem:corner-perfect}
Let $A$ be a finite-dimensional $\kk$-algebra and $e\in A$ an
idempotent.  Regard $eA$ as an $(eAe,A)$-bimodule.  Then
\[
 -\otimes_{eAe}^{\mathbf L}eA\colon
 \perf(eAe)\longrightarrow\perf(A)
\]
is fully faithful, and its essential image is
$\thick_{\perf(A)}(eA)$.
\end{lemma}

\begin{proof}
The module $eA$ is a direct summand of the regular right $A$-module, so
it is perfect over $A$.  On the compact generator $eAe$ of
$\perf(eAe)$, the functor takes the value $eA$, and
\[
 \RHom_A(eA,eA)=\Hom_A(eA,eA)\cong eAe.
\]
The higher derived Hom groups vanish because $eA$ is projective as a right
$A$-module.  Hence the functor induces isomorphisms on all graded
morphism spaces between finite direct sums of shifts of the generator.
The class of objects on which these Hom maps are isomorphisms
is closed under shifts, cones, and direct summands in each variable.
This proves full faithfulness on $\perf(eAe)$. Finally, because $\perf(eAe)=\thick(eAe)$, the essential
image is exactly $\thick_{\perf(A)}(eA)$.
\end{proof}

\begin{lemma}
\label{lem:reflected-corner}
Let $\sigma_N(z_1,\ldots,z_m)
 =(N+1-z_m,\ldots,N+1-z_1)$.
For every $\Omega\subseteq\os_m(N)$, coordinate reversal induces an isomorphism
\[
 \bigl(e_\Omega\mathsf A_{N,m}e_\Omega\bigr)^{\op}
 \cong
 e_{\sigma_N(\Omega)}\mathsf A_{N,m}
 e_{\sigma_N(\Omega)}.
\]
\end{lemma}

\begin{proof}
For $x,y\in\os_m(N)$, the inequalities $x\preceq y$ are equivalent,
after subtracting every entry from $N+1$ and reversing the order, to
$\sigma_N(y)\preceq\sigma_N(x)$.  Thus a basis morphism $x\to y$
is sent to a basis morphism
$\sigma_N(y)\to\sigma_N(x)$.  Applied to three vertices, the same
calculation shows that the outer pair interlaces before reflection if and
only if the reflected outer pair interlaces.  The rule therefore preserves
both cases of the multiplication formula
\eqref{eq:higher-composition}, including its zero products.  It is an
anti-isomorphism of the full higher Auslander category; restricting the
objects to $\Omega$ gives the claimed corner isomorphism.
\end{proof}

\begin{theorem}
\label{thm:staircase-fukaya-corner}
Let $1\leq m\leq N$, and let
$\Omega\subseteq\os_m(N)$ be the nonempty coordinate interval $\Omega(\ell,H)$.
Put $\overline\Omega=\sigma_N(\Omega)$.
Then there is a triangle equivalence
\begin{equation}
 \Db{\kk\Omega}\simeq\thick_{H^0(\Wcat_N^{(m)})}
 \{L_I:I\in\overline\Omega\}.
\label{eq:staircase-fukaya-corner}
\end{equation}
\end{theorem}

\begin{proof}

Let $A=\mathsf A_{N,m}$. Theorem~\ref{thm:coordinate-intervals}
and Proposition~\ref{prop:corner} give $\Db{\kk\Omega}
 \simeq\Db{e_{\Omega}Ae_{\Omega}}$.
Since $\kk\Omega$ has finite global dimension, so does the
derived-equivalent corner. Its bounded derived category therefore
agrees with its perfect category.
Our covariant-module convention identifies the bounded derived category of
left $e_{\Omega}Ae_{\Omega}$-modules with the perfect category of
right $(e_{\Omega}Ae_{\Omega})^{\op}$-modules.  By Lemma
\ref{lem:reflected-corner}, $(e_{\Omega}Ae_{\Omega})^{\op}
 \cong
 e_{\overline\Omega}Ae_{\overline\Omega}$.
Lemma \ref{lem:corner-perfect} identifies the perfect category of this
corner with $\thick_{\perf(A)}
 \{e_IA:I\in\overline\Omega\}$.
Finally, the $H^0$-level of the first quasi-equivalence in
\eqref{eq:DJL-equivalences} sends $e_IA$ to the standard product
Lagrangian $L_I$.  This proves \eqref{eq:staircase-fukaya-corner}.
\end{proof}

\begin{corollary}
\label{thm:dyck-fukaya-corner}
Let $a,b\geq2$, and set
\begin{equation}
 s=\left\lceil\frac ab\right\rceil,\qquad
 m=a-s,\qquad N=a+b-s,
\label{eq:fukaya-local-indices}
\end{equation}
Define
 $\sigma_N(z_1,\ldots,z_m)=(N+1-z_m,\ldots,N+1-z_1)$
and $\overline\Omega_{a,b}=\sigma_N(\Omega_{a,b})$.  There is a triangle
equivalence
\begin{equation*}
 \Db{\kk\Dyck^{\mathrm{below}}_{a,b}}
 \simeq
 \thick_{H^0(\Wcat_N^{(m)})}
 \{L_I:I\in\overline\Omega_{a,b}\}.
\end{equation*}
\end{corollary}

\begin{proof}
By Proposition \ref{prop:dyck-staircase}, the below-diagonal Dyck poset is
isomorphic to the coordinate staircase $\Omega_{a,b}\subseteq\os_m(N)$.
Apply Theorem \ref{thm:staircase-fukaya-corner} to this staircase.
\end{proof}

\begin{remark}
For any coordinate interval $\Omega$, if one works throughout with right modules
instead of the covariant functors used here, the indexing set in
\eqref{eq:staircase-fukaya-corner} is \rev{$\Omega$} rather than its reflection
\rev{$\overline\Omega$}.  Thus the reflection records a module-side convention,
not a geometric difference.
\end{remark}

\subsection{A symplectic form of the CLR equivalence}

\begin{theorem}
\label{thm:symplectic-CLR}
Let $a,b$ be coprime positive integers and put $r=a+b$.  There are
triangle equivalences
\begin{equation}
\begin{split}
 \Db{\kk(\ovr A_r\times\Dyck^{\mathrm{above}}_{a,b})}
 &\simeq\Db{\kk L_{a,b}}\\
 &\simeq H^0(\Wcat_r^{(a)})
 \simeq H^0(\Wcat_r^{(b)}).
\end{split}
\label{eq:symplectic-CLR}
\end{equation}
\end{theorem}

\begin{proof}
The first equivalence is Corollary~\ref{cor:CLR-conjecture}.  By
\eqref{eq:rectangle-equivalence} and the index conversion
\eqref{eq:DJL-index},
$\Db{\kk L_{a,b}}
 \simeq\Db{A_{b+1}^{a}}
 =\Db{\mathsf A_{r,a}}$.
The algebra $\mathsf A_{r,a}$ has finite global dimension and is
isomorphic to its opposite by Lemma \ref{lem:A-self-op}.  Hence its bounded
derived category in our left-module convention agrees with
$\perf(\mathsf A_{r,a})$ in the right-module convention of
\cite{DJL}.  The remaining two equivalences are
\eqref{eq:DJL-equivalences}, since $r-a=b$.
\end{proof}

Thus the transposition symmetry of the $a\times b$ rectangle is realized
symplectically by changing the symmetric-product degree from $a$ to
$b=r-a$.

\subsection{Fukaya--Seidel and Serre-functor consequences}

\begin{lemma}
\label{lem:transport-serre}
Let $\mathcal D$ and $\mathcal E$ be Hom-finite
$\kk$-linear triangulated categories with Serre functors
$S_{\mathcal D}$ and $S_{\mathcal E}$.  If
$F\colon\mathcal D\to\mathcal E$ is a triangle equivalence, then
$F S_{\mathcal D}\cong S_{\mathcal E}F$.
Consequently, an isomorphism $S_{\mathcal D}^{q}\cong[p]$, or its
twisted version $S_{\mathcal D}^{q}\cong\phi[p]$, transports to the
corresponding isomorphism on $\mathcal E$.
\end{lemma}

\begin{proof}
The functor $FS_{\mathcal D}F^{-1}$ represents
$D\Hom_{\mathcal E}(-,-)$ by Serre duality and full faithfulness.
Uniqueness of the Serre functor identifies it naturally with
$S_{\mathcal E}$; iteration transports the stated relations.
\end{proof}

Di Dedda \cite[Corollary~1.4]{DiDedda} considers the symmetric Brieskorn--Pham
singularity
\begin{equation*}
 f_{n,d}:\operatorname{Sym}^{d}(\mathbb C)\cong\mathbb C^d
 \longrightarrow\mathbb C,\qquad
 \{x_1,\ldots,x_d\}\longmapsto
 x_1^{n+1}+\cdots+x_d^{n+1},
\end{equation*}
and proves the first of the following quasi-equivalences of pretriangulated
$A_\infty$-categories; the second is \cite[Theorem~1]{DJL}:
\begin{equation*}
 \Fcat(f_{n,d})
 \simeq\operatorname{Perf}_{\mathrm{dg}}(\mathsf A_{n,d})
 \simeq\Wcat_n^{(d)}.
\end{equation*}
Consequently, Theorem \ref{thm:symplectic-CLR} yields
 $\Db{\kk L_{a,b}}
 \simeq H^0\Fcat(f_{a+b,a})$,
and the same Fukaya--Seidel category models the left-hand side of the CLR
equivalence.

\begin{proposition}
\label{prop:staircase-FS-admissible}
Let $\Omega\subseteq\os_m(N)$ be the nonempty coordinate interval $\Omega(\ell,H)$,
where $1\leq m\leq N$. There is a distinguished collection of graded
thimbles $(\Delta_I)_{I\in\os_m(N)}$ for $f_{N,m}$ such that
\[
 \Db{\kk\Omega}\simeq
 \thick_{H^0\Fcat(f_{N,m})}
 \{\Delta_I:I\in\sigma_N(\Omega)\}.
\]
This subcategory and the subcategory in
\eqref{eq:staircase-fukaya-corner} are admissible: their inclusions have
both left and right adjoints.

More explicitly, put $A=\mathsf A_{N,m}$,
$e=e_{\sigma_N(\Omega)}$, and $B=eAe$.
Under either model, the inclusion is
$j=-\otimes_B^{\mathbf L}eA$, its right adjoint is $j^!(X)=Xe$,
and the intrinsic Serre functor is
$S_B\cong j^!S_Aj\cong-\otimes_B^{\mathbf L}DB$.
The right projection onto its image is
$\Pi(X)=(Xe)\otimes_B^{\mathbf L}eA$.
\end{proposition}

\begin{proof}
Di Dedda constructs, by Hurwitz moves, a distinguished thimble collection
whose derived endomorphism algebra is quasi-isomorphic to
$\mathsf A_{N,m}$, concentrated in degree zero
\cite[Proposition~3.52]{DiDedda}. Use this algebra identification to label the primitive summands. Derived Yoneda then sends $\Delta_I$ to the
right projective $e_IA$. This fixes the labels independently of the
left-module convention used elsewhere in \cite{DiDedda}.
Lemma~\ref{lem:corner-perfect}, the reflected-corner isomorphism, and
Theorem~\ref{thm:coordinate-intervals} give the asserted equivalence.

Both $A$ and $B$ have finite global dimension. The usual tensor--Hom
adjunction gives $j^!=\RHom_A(eA,-)=(-)e$; it takes perfect objects to
perfect objects because $B$ has finite global dimension. The left
adjoint is $j^*=S_B^{-1}j^!S_A$, as follows directly from Serre duality.
This proves admissibility. For $Y\in\perf(B)$, associativity of derived
tensor products and the bimodule isomorphism $e(DA)e\cong DB$ give
\[
 j^!S_Aj(Y)
 \cong\bigl(Y\otimes_B^{\mathbf L}eA
                 \otimes_A^{\mathbf L}DA\bigr)e
 \cong Y\otimes_B^{\mathbf L}DB.
\]
Finally, $\Pi=jj^!$. The counit triangle
$\Pi X\longrightarrow X\longrightarrow Q_X\longrightarrow\Pi X[1]$
satisfies $Q_Xe=0$, so $Q_X$ belongs to the right orthogonal of the
staircase subcategory. Transporting these functors along the equivalences
gives the same statements in both geometric models.
\end{proof}

The proposition applies to every skew Ferrers diagram and rational
Dyck staircase, without coprimality or the window tensor equivalence. It realizes
the interval category by selected thimbles in the ambient
Fukaya--Seidel category. A model as the entire category of a new
singularity, or embedded representatives for all intermediate tilting
generators, would require additional geometric constructions.

The Serre functor also becomes geometric.  Dyckerhoff--Jasso--Lekili
\cite[Proposition~3.5.1 and the following remark]{DJL} show
that a graded lift of the rotation of $D$ through
$2\pi/(n+1)$ induces the Serre functor $S$ on
$H^0(\Wcat_n^{(d)})$, and that
\begin{equation}
 S^{n+1}\cong[d(n-d)].
\label{eq:DJL-Serre}
\end{equation}
For $n=a+b$ and $d=a$, this gives
\begin{equation}
 S^{a+b+1}\cong[ab]
\label{eq:lattice-Serre}
\end{equation}
on every category in \eqref{eq:symplectic-CLR}. For the full
higher Auslander algebra, this is precisely
\eqref{eq:full-algebra-period}, used earlier independently of the CLR corollary.
Thus the full path lattice is fractionally Calabi--Yau of
dimension $ab/(a+b+1)$, as reflected by the graded disk rotation.
The local Dyck corner has the distinct Serre relation in
Corollary~\ref{cor:local-dyck-serre} below. The passage of the Serre relation
through the algebraic and symplectic equivalences is justified by
Lemma \ref{lem:transport-serre}; no identification of Serre functors is being
assumed implicitly.

\begin{lemma}
\label{lem:serre-tensor-cancellation}
Let $C$, $B$, and $C\otimes_\kk B$ be finite-dimensional algebras of
finite global dimension. Suppose that $E\in\perf(C)$ is exceptional,
meaning $\RHom_C(E,E)\simeq\kk$, and that
$S_C^qE\cong E[u]$ for some $q>0$ and $u\in\mathbb Z$.
If $S_{C\otimes B}^q\cong[v]$, then $S_B^q\cong[v-u]$.
\end{lemma}

\begin{proof}
The external tensor functor $F=E\otimes_\kk-\colon
\perf(B)\to\perf(C\otimes B)$ is fully faithful: the K\"unneth
isomorphism identifies its derived Hom complexes with
$\RHom_C(E,E)\otimes_\kk\RHom_B(-,-)$.
For right modules, Serre functors are derived tensor products with the
dual bimodules. Since $D(C\otimes B)\cong DC\otimes DB$, iteration gives
a natural isomorphism
\[
 S_{C\otimes B}^qF(-)
 \cong S_C^qE\otimes_\kk S_B^q(-)
 \cong F S_B^q(-)[u].
\]
The hypothesis therefore gives $F S_B^q\cong F[v-u]$.
Full faithfulness lifts each component and its inverse uniquely
to $\perf(B)$. Faithfulness then gives naturality of the lifted
maps, proving the functor isomorphism.
\end{proof}

\begin{corollary}
\label{cor:local-dyck-serre}
Let $a,b\geq2$ be coprime. The Serre functor of each of
$\Db{\kk\Dyck^{\mathrm{below}}_{a,b}}$,
$\Db{\kk\Dyck^{\mathrm{above}}_{a,b}}$, and $\mathcal K_{a,b}$ satisfies
$S^{a+b+1}\cong[(a-1)(b-1)]$.
Here $s=\lceil a/b\rceil$ and
$\mathcal K_{a,b}:=
 \thick_{H^0(\Wcat_{a+b-s}^{(a-s)})}
 \{L_I:I\in\overline\Omega_{a,b}\}$.
Thus these categories are fractionally Calabi--Yau of dimension
$(a-1)(b-1)/(a+b+1)$. No minimality of the displayed exponent is claimed.
\end{corollary}

\begin{proof}
Put $r=a+b$ and let $B$ be the opposite of the incidence algebra of
$\Dyck^{\mathrm{above}}_{a,b}$, so that its perfect right-module category
is the left-module category in the statement. Corollary~\ref{cor:CLR-conjecture}
and \eqref{eq:lattice-Serre} give
$S_{\kk\ovr A_r\otimes B}^{r+1}\cong[ab]$; here we use
$(\kk\ovr A_r)^{\op}\cong\kk\ovr A_r$ to match module conventions.
For the chain algebra $C=\kk\ovr A_r$, formula
\eqref{eq:DJL-Serre} with $d=1$ gives $S_C^{r+1}\cong[r-1]$.
Any indecomposable projective of $C$ is exceptional. Applying
Lemma~\ref{lem:serre-tensor-cancellation} yields
$S_B^{r+1}\cong[ab-(r-1)]=[(a-1)(b-1)]$.
Vector-space duality takes the Serre functor to the inverse Serre functor
and $[p]$ to $[-p]$. Taking opposites therefore preserves this relation,
which proves the below-diagonal case by Lemma~\ref{lem:above-below}.
Corollary~\ref{thm:dyck-fukaya-corner} and
Lemma~\ref{lem:transport-serre} give the assertion for $\mathcal K_{a,b}$.
\end{proof}

\begin{remark}
The corner relation for $B_0$ is known
\cite[Proposition~4.34(ii)]{Xing}. Here the Dyck relation follows
from the CLR corollary and tensor cancellation. Its derived
equivalence is the composite of diagram tilting and the direct
window tensor equivalence of Corollary~\ref{cor:direct-window-tensor}.
The additional window homology is a separate consequence,
Corollary~\ref{cor:window-homology}.

Proposition~\ref{prop:staircase-FS-admissible} also gives a precise
geometric interpretation. If $R$ is the ambient graded disk rotation
inducing $S_A$, and $i\colon\mathcal K_{a,b}\hookrightarrow
H^0(\Wcat_N^{(m)})$ is the inclusion, then
$iS_{\mathcal K_{a,b}}\cong\Pi Ri$. Consequently,
$(\Pi R)^{a+b+1}i\cong i[(a-1)(b-1)]$.
The projection is applied after every rotation. No preservation of the
staircase subcategory by $R$ alone is required.
\end{remark}

\subsection{Categorical refinement and a geometric lifting problem}

For the remainder of this section fix a nonempty coordinate interval
$\Omega=\Omega(\ell,H)\subseteq\os_m(N)$, and put
$\overline\Omega=\sigma_N(\Omega)$ and
$\mathcal K_\Omega
 =
 \thick_{H^0(\Wcat_N^{(m)})}
 \{L_I:I\in\overline\Omega\}$.
For a Dyck staircase $\Omega=\Omega_{a,b}$ with the indices in
\eqref{eq:fukaya-local-indices}, we also write
$\mathcal K_{a,b}=\mathcal K_\Omega$.
By Theorem~\ref{thm:staircase-fukaya-corner}, this category is triangle equivalent to
each of the staircase categories
$\Db{\C_j(\Omega)}$.  We now choose these equivalences compatibly
and identify the images of their standard projective generators.

Let $\widetilde{\mathcal K}_\Omega$ be the full pretriangulated and
idempotent-complete $A_\infty$-subcategory of $\Wcat_N^{(m)}$ split-generated
by the product Lagrangians $L_I$ with
$I\in\overline\Omega$.  Thus
 $H^0(\widetilde{\mathcal K}_\Omega)=\mathcal K_\Omega$.
For every $j$, we use
 $\operatorname{Perf}_{\mathrm{dg}}
 \!\left(\operatorname{Alg}(\C_j(\Omega))^{\op}\right)$
as the dg enhancement of $\Db{\C_j(\Omega)}$.  Their cohomological
categories agree because these category algebras have finite global
dimension by the staircase construction.

For $v\in\Omega$, let $P_v^{(j)}
 =
 \Hom_{\C_j(\Omega)}(v,-)$.
For $1\leq j\leq m-1$, write $v=(p,t)$ according to the prefix--tail
decomposition used in the proof of Theorem~\ref{thm:coordinate-intervals}, and put
$T_v^{(j)}
 =
 (\iota_t)_*P_p,
 T^{(j)}
 =
 \bigoplus_{v\in\Omega}T_v^{(j)}$.
Here $P_p=\Hom_{F_j(t)}(p,-)$, and $(\iota_t)_*$ is the exact right
Kan extension appearing in Lemma \ref{lem:kan}.  The $j$-th staircase
equivalence is
\begin{equation*}
 \mathsf F_j
 =
 \RHom_{\C_{j-1}(\Omega)}
 \bigl(T^{(j)},-\bigr)
 \colon
 \Db{\C_{j-1}(\Omega)}
 \longrightarrow
 \Db{\C_j(\Omega)}.
\end{equation*}
It has a standard dg lift
\begin{equation*}
 \widetilde{\mathsf F}_j\colon
 \operatorname{Perf}_{\mathrm{dg}}
 \!\left(\operatorname{Alg}(\C_{j-1}(\Omega))^{\op}\right)
 \longrightarrow
 \operatorname{Perf}_{\mathrm{dg}}
 \!\left(\operatorname{Alg}(\C_j(\Omega))^{\op}\right),
\end{equation*}
given by the derived Hom dg bimodule associated with $T^{(j)}$; its
cohomological functor is $\mathsf F_j$.  Choose an enhanced
quasi-equivalence
\begin{equation}
 \widetilde\Phi_{m-1}\colon
 \operatorname{Perf}_{\mathrm{dg}}
 \!\left(\operatorname{Alg}(\C_{m-1}(\Omega))^{\op}\right)
 \longrightarrow\widetilde{\mathcal K}_\Omega
\label{eq:enhanced-final-Phi}
\end{equation}
by composing the dg corner embedding, the reflection isomorphism of
Lemma \ref{lem:reflected-corner}, and the quasi-equivalence
\eqref{eq:DJL-equivalences}.  Recursively define
\begin{equation}
 \widetilde\Phi_{j-1}
 =\widetilde\Phi_j\circ\widetilde{\mathsf F}_j
 \qquad(1\leq j\leq m-1),
\label{eq:enhanced-Phi-recursion}
\end{equation}
and put $\Phi_j=H^0(\widetilde\Phi_j)$.  These choices make all the
equivalences below compatible at the enhanced level.
For $v\in\Omega$, set
 $\widetilde{\mathbb L}_v^{(j)}
 =\widetilde\Phi_j(P_v^{(j)})$ in $\widetilde{\mathcal K}_\Omega$,
and write $\mathbb L_v^{(j)}$ for the same object viewed in
$H^0(\widetilde{\mathcal K}_\Omega)=\mathcal K_\Omega$.  Put
$\widetilde{\mathbb L}^{(j)}=
\{\widetilde{\mathbb L}_v^{(j)}:v\in\Omega\}$.

\begin{theorem}
\label{thm:categorical-fukaya-staircase}
For every $0\leq j\leq m-1$, there is a triangle equivalence
\begin{equation*}
 \Phi_j\colon
 \Db{\C_j(\Omega)}
 \longrightarrow
 \mathcal K_\Omega
\end{equation*}
and a collection $\mathbb L^{(j)}
 =
 \{\mathbb L_v^{(j)}:v\in\Omega\}
 \subseteq\mathcal K_\Omega$
with the following properties.
\begin{enumerate}
\item The collection $\mathbb L^{(j)}$ split-generates
$\mathcal K_\Omega$.
\item For $u,v\in\Omega$,
\begin{equation}
 \Hom_{\mathcal K_\Omega}
 \bigl(\mathbb L_u^{(j)},\mathbb L_v^{(j)}[r]\bigr)
 =
 \begin{cases}
  \Hom_{\C_j(\Omega)}(v,u),&r=0,\\
  0,&r\neq0.
 \end{cases}
\label{eq:L-j-Hom}
\end{equation}
Consequently,
\begin{equation}
 \End_{\mathcal K_\Omega}
 \left(
  \bigoplus_{v\in\Omega}\mathbb L_v^{(j)}
 \right)^{\op}
 \cong
 \operatorname{Alg}\!\left(\C_j(\Omega)\right).
\label{eq:L-j-end}
\end{equation}
Moreover, the full $A_\infty$-subcategory on chosen enhanced representatives
$\widetilde{\mathbb L}^{(j)}$ of $\mathbb L^{(j)}$ is formal and is
quasi-equivalent to
$\C_j(\Omega)^{\op}$, regarded as an $A_\infty$-category
concentrated in degree zero.
\item For the final staircase category, the collection can be chosen as
\begin{equation}
 \mathbb L_v^{(m-1)}
 \cong
 L_{\sigma_N(v)}.
\label{eq:final-product-lagrangian}
\end{equation}
\item For $1\leq j\leq m-1$, one has
\begin{equation}
 \mathbb L_v^{(j)}
 \cong
 \Phi_{j-1}\bigl(T_v^{(j)}\bigr).
\label{eq:Kan-Fukaya-image}
\end{equation}
Thus the passage from the generating collection
$\mathbb L^{(j-1)}$ to $\mathbb L^{(j)}$ is precisely the
categorical change of generators induced by the Kan-extension tilting
object $T^{(j)}$.
\end{enumerate}
\end{theorem}

\begin{proof}
Use the enhanced functors fixed in
\eqref{eq:enhanced-final-Phi}--\eqref{eq:enhanced-Phi-recursion}.  The
module-side convention and the reflection $\sigma_N$ imply $\Phi_{m-1}(P_v^{(m-1)})
 \cong L_{\sigma_N(v)}$.
The projectives $P_v^{(j)}$ split-generate
$\Db{\C_j(\Omega)}$, and therefore their images split-generate
$\mathcal K_\Omega$.

Since $P_u^{(j)}$ is projective,
$\Hom_{\Db{\C_j}}
 \bigl(P_u^{(j)},P_v^{(j)}[r]\bigr)=0$ for 
 $r\neq0$.
For $r=0$, Yoneda's lemma gives $
 \Hom_{\C_j\modcat}
 \bigl(P_u^{(j)},P_v^{(j)}\bigr)
 \cong
 \Hom_{\C_j}(v,u)$.
Applying the fully faithful functor $\Phi_j$ proves
\eqref{eq:L-j-Hom} and \eqref{eq:L-j-end}.

Let $\widetilde{\mathcal E}^{(j)}$ be the full, strictly unital and
uncurved $A_\infty$-subcategory of $\widetilde{\mathcal K}_\Omega$ on the
objects $\widetilde{\mathbb L}_v^{(j)}$.  Its cohomology is concentrated in
degree zero by \eqref{eq:L-j-Hom}.  Pass to a strictly unital minimal model.
Here $\mu^0=0$ because the category is uncurved and $\mu^1=0$ by
minimality.  For $q\geq3$, the operation $\mu^q$ has degree $2-q$ and
therefore vanishes: all inputs have degree zero, whereas the putative output
would have negative degree.  The remaining operation $\mu^2$ is the composition in the
indexed category $\C_j(\Omega)^{\op}$, because
$H^0(\widetilde\Phi_j)=\Phi_j$ is fully faithful and preserves the Yoneda
composition between the labeled projectives $P_v^{(j)}$.  Thus the minimal
model is $\C_j(\Omega)^{\op}$, proving formality and the asserted
quasi-equivalence.

It remains to verify the compatibility with the Kan-extension summands.
By Proposition \ref{prop:end-relative}, the opposite of the full subcategory
on the summands of $T^{(j)}$ is $\C_j(\Omega)$.  Under this
identification,
\begin{equation*}
 \mathsf F_j(T_v^{(j)})
 =
 \RHom_{\C_{j-1}}(T^{(j)},T_v^{(j)})
 \cong
 P_v^{(j)}.
\end{equation*}
Taking $H^0$ in \eqref{eq:enhanced-Phi-recursion}, we obtain $\Phi_{j-1}(T_v^{(j)})=
 \Phi_j\mathsf F_j(T_v^{(j)})
 \cong
 \Phi_j(P_v^{(j)})
 =\mathbb L_v^{(j)}$.
This proves \eqref{eq:Kan-Fukaya-image}.
\end{proof}

\begin{corollary}
Every object $\mathbb L_v^{(j)}$ admits a finite twisted-complex
presentation whose terms are standard product Lagrangians
$L_I, I\in\overline\Omega$.
Finite projective resolutions and the successive functors
$\mathsf F_{j+1},\ldots,\mathsf F_{m-1}$ compute this presentation.
\end{corollary}

\begin{proof}
Every $\C_j(\Omega)$ is finite and directed and therefore has finite
global dimension.  Hence every bounded complex obtained while applying the
successive derived Hom functors is perfect over
$\C_{m-1}(\Omega)$.  It is consequently represented by a finite
complex of the projectives $P_w^{(m-1)}$.  Under $\Phi_{m-1}$, these
projectives become the standard product Lagrangians
$L_{\sigma_N(w)}$ by \eqref{eq:final-product-lagrangian}.  The resulting
finite complex is the required twisted-complex presentation.
\end{proof}

\begin{remark}
\Cref{thm:categorical-fukaya-staircase} identifies the specific
Kan-extension tilting functors inside the stopped symmetric-product Fukaya
category.  It does not assert that the twisted complexes
$\mathbb L_v^{(j)}$ are represented by single embedded Lagrangian branes.
\end{remark}

\begin{Prob}
Construct exact graded embedded representatives of
$\mathbb L_v^{(j)}$ for $0\leq j\leq m-3$ when $m\geq3$,
and realize all changes of generators from
$\mathbb L^{(j-1)}$ to $\mathbb L^{(j)}$ by explicit Lagrangian
mutations or stop-compatible handle slides.
\end{Prob}

A geometric solution should compute the Floer-chain products and
differentials responsible for the half-square relations, while keeping
track of the stop.  Vanishing in cohomology does not in general mean
that the relevant holomorphic triangles are absent: already for the
standard product generators below, a nonzero triangle product is the
boundary of a crossing generator.

The next theorem combines three pieces of the construction:
cubical projective resolutions, embedded representatives at the last
tilting step, and chain-level null-homotopies for half-square products.
All three apply to coordinate intervals; the embedded-representative
statement concerns the last step in degree $m\geq2$.

\begin{theorem}\label{thm:cubical-geometric-tilting}
Let $m\geq2$, let $\ell,H\in\os_m(N)$ with $\ell\leq H$, and put
$\Omega=\Omega(\ell,H)$. A coordinate staircase is the special
case $\ell=(1,\ldots,m)$ with effective upper bounds $H$.
Choose $\Phi_{m-1}(P_z^{(m-1)})\cong L_{\sigma_N(z)}$ and
define $\Phi_{j-1}=\Phi_j\mathsf F_j$ as above.
\begin{enumerate}
\item For $1\leq j<m$ write $v=(p,t)\in\Omega$, with
$p=(p_1,\ldots,p_j)$ and $t=(t_1,\ldots,t_h)$, $h=m-j$.
For $S\subseteq\{1,\ldots,h\}$ define $d_0^S=p_j$ and
\[
 d_q^S=
 \begin{cases}
 \max\{\ell_{j+q},d_{q-1}^S+1,t_q+1\},&q\in S,\\
 \max\{\ell_{j+q},d_{q-1}^S+1\},&q\notin S,
 \end{cases}
 \qquad b_S=(p,d_1^S,\ldots,d_h^S).
\]
Interpret $P_{b_S}^{(j-1)}$ as zero if $b_S\notin\Omega$. Then
\begin{equation}
 0\longrightarrow\bigoplus_{|S|=h}P_{b_S}^{(j-1)}
 \longrightarrow\cdots\longrightarrow
 \bigoplus_{|S|=1}P_{b_S}^{(j-1)}
 \longrightarrow P_{b_\varnothing}^{(j-1)}
 \longrightarrow T_v^{(j)}\longrightarrow0
\label{eq:kan-cubical-resolution}
\end{equation}
is a projective resolution. Its differential deletes one element
of $S$, with alternating signs and the natural inclusions.
In particular, $\operatorname{pd}T_v^{(j)}\leq m-j$.
Under the enhanced $\Phi_{j-1}$, this gives a coherent cubical
presentation of $\mathbb L_v^{(j)}$ with vertices
$\mathbb L_{b_S}^{(j-1)}$.
\item At the last step, write $p=(p_1,\ldots,p_{m-1})$ and put
\begin{align*}
 s_0(p)&=\max\{\ell_m,p_{m-1}+1\},\qquad u=N+1-H_m,\\
 V_p&=\prod_{r=1}^{m-1}\lambda_{0,N+1-p_r}.
\end{align*}
For every $(p,s)\in\Omega$, suitable gradings give
\begin{equation*}
 \mathbb L_{p,s}^{(m-2)}\cong
 \begin{cases}
  \lambda_{0,u}\times V_p,&s=s_0(p),\\
  (\lambda_{u,N+2-s}\times V_p)[-1],&s>s_0(p).
 \end{cases}
\end{equation*}
All these objects are represented by exact graded embedded product
Lagrangians. For $m=2$ this describes the entire incidence collection.
\item In the standard Floer-chain model, let $I\leq J\leq K$ be
product-brane labels satisfying
\[
 j_r<i_{r+1},\quad k_r<j_{r+1}\quad(1\leq r<m),
 \qquad k_a\geq i_{a+1}\quad\text{for some }a<m.
\]
For the noncrossing degree-zero generators
$f\in CF^0(L_I,L_J)$ and $g\in CF^0(L_J,L_K)$, one has
\begin{equation*}
 \begin{aligned}
 \mu^2(g,f)&=\varepsilon e_{KI},\qquad\varepsilon\in\{1,-1\},\\
 e_{KI}&=\mu^1(h),\qquad h\in CF^{-1}(L_I,L_K).
 \end{aligned}
\end{equation*}
Thus the half-square product is a boundary, despite its nonzero
chain-level triangle coefficient.
\end{enumerate}
\end{theorem}

\begin{proof}
For (1), $b_\varnothing\leq v$, and its construction satisfies
the lower bounds, so $b_\varnothing\in\Omega$.
If $S\subseteq S'$, the representable at $b_{S'}$ is naturally
a submodule of that at $b_S$. Their inclusions commute, proving
that the signed differentials square to zero. Evaluate at
$z=(a,c)\in\Omega$. If $\Hom_{\X_j}(p,a)=0$, all terms vanish.
Otherwise the term indexed by $S$ is one-dimensional exactly when
$c_q>t_q$ for every $q\in S$: the lower bounds and strict increase
of $c$ already supply the other constraints in the recursion.
If $b_S$ is outside $\Omega$, no such $z$ can dominate it;
indeed $b_S$ already satisfies the lower bounds and domination
would supply its upper bounds. Thus the nonzero terms are precisely
the subsets of $B(z)=\{q:c_q>t_q\}$.
When $B(z)=\varnothing$, the augmentation is the identity.
Otherwise the Kan summand is zero at $z$ and the complex is the
augmented simplicial chain complex on $B(z)$, which is exact over
any field. This proves the resolution pointwise. The enhanced
equivalences transport its total complex to the stated cubical
presentation.

For (2), fix $p$, put $s_0=s_0(p)$, and apply (1) with $h=1$.
For $s_0\leq t<H_m$ this gives
\begin{equation*}
 0\longrightarrow P_{p,t+1}^{(m-2)}
 \longrightarrow P_{p,s_0}^{(m-2)}
 \longrightarrow T_{p,t}^{(m-1)}
 \longrightarrow0,\qquad
 T_{p,H_m}^{(m-1)}=P_{p,s_0}^{(m-2)}.
\end{equation*}
Reflection and $\mathsf F_{m-1}(T_{p,t})=P_{p,t}^{(m-1)}$
identify the middle term's image with $\lambda_{0,u}\times V_p$
and the quotient's image with $\lambda_{0,w}\times V_p$, where
$w=N+1-t$. The quotient map is nonzero and the Hom between its
images is one-dimensional. Since
$0<u<w<N+1-p_{m-1}$, the product-arc triangle of
\cite[Lemma~3.3.4(2)]{DJL} applies:
\begin{equation*}
 \lambda_{0,u}\times V_p\longrightarrow
 \lambda_{0,w}\times V_p\longrightarrow
 \lambda_{u,w}\times V_p\longrightarrow
 (\lambda_{0,u}\times V_p)[1].
\end{equation*}
The common arcs can be chosen pairwise disjoint, and none is
homotopic to any of the three varying arcs. The non-split triangle and the
one-dimensional Hom identify its first map, up to a nonzero scalar,
with the quotient map. Rotating the triangle and setting $s=t+1$
proves the formula. All common endpoints lie beyond $w$, so the
arcs in each displayed product have disjoint representatives; their
products are embedded, and the shifts specify the gradings.

For (3), \cite[Proposition~3.2.5]{DJL} identifies the noncrossing
Floer-chain subalgebra with the incidence algebra, up to orientation
signs. Hence the product is $\varepsilon e_{KI}$.
The outer interlacing failure gives $HF^*(L_I,L_K)=0$ by
\cite[Proposition~3.2.8 and the proof of Theorem~3.2.3]{DJL}. The standard graded complex has no
positive-degree generators, so its degree-zero element $e_{KI}$
is closed and must be the differential of a degree-minus-one chain.
The standard Floer model supplies the strip differential and
triangle product used here. Specifying a particular strip domain
near the stop would require a further geometric calculation.
\end{proof}

The cubical resolution and the Floer boundary statement describe
different differentials: the former resolves a Kan module, while the
latter supplies the null-homotopies for products in the standard
geometric model. Together with the one-coordinate product-arc triangles, they
specify the algebraic maps and homotopies required for a geometric
realization. They do not establish embedded representatives for
all earlier layers or a simultaneous handle-slide order.

\section{Examples of the algebraic and geometric constructions}
\label{sec:examples}

\subsection{The \texorpdfstring{$(3,4)$}{(3,4)} staircase}

For $(a,b)=(3,4)$, the inequalities in \eqref{eq:dyck-bound} give
$x_1\leq1, x_2\leq3, x_3\leq5$.
Hence
\begin{equation*}
 \Dyck^{\mathrm{below}}_{3,4}
 =
 \{(1,2,3),(1,2,4),(1,2,5),(1,3,4),(1,3,5)\}.
\end{equation*}
The five paths are displayed in Figure \ref{fig:34-below-paths}.  The dashed
line is the diagonal $q=4p/3$, and every blue path remains weakly below it.
Here $s=\lceil3/4\rceil=1$, so deleting the forced first coordinate and
subtracting $1$ from the remaining coordinates gives
$\beta(x_1,x_2,x_3)=(x_2-1,x_3-1)$.

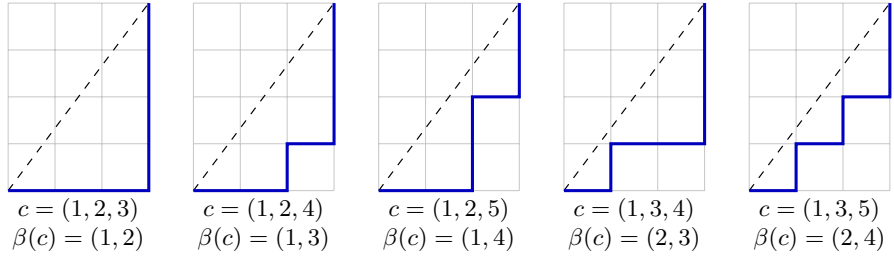
\begin{figure}[htbp]
\centering
\begin{tikzpicture}[
 x=0.62cm,y=0.62cm,
 pathline/.style={blue!75!black,line width=1.15pt},
 diagonal/.style={black,dashed,line width=0.45pt},
 gridline/.style={gray!45,line width=0.25pt}]

 \begin{scope}[xshift=0cm]
  \draw[step=1,gridline] (0,0) grid (3,4);
  \draw[diagonal] (0,0)--(3,4);
  \draw[pathline] (0,0)--(3,0)--(3,4);
  \node[align=center,font=\scriptsize] at (1.5,-0.72)
   {$c=(1,2,3)$\\$\beta(c)=(1,2)$};
 \end{scope}

 \begin{scope}[xshift=2.45cm]
  \draw[step=1,gridline] (0,0) grid (3,4);
  \draw[diagonal] (0,0)--(3,4);
  \draw[pathline] (0,0)--(2,0)--(2,1)--(3,1)--(3,4);
  \node[align=center,font=\scriptsize] at (1.5,-0.72)
   {$c=(1,2,4)$\\$\beta(c)=(1,3)$};
 \end{scope}

 \begin{scope}[xshift=4.90cm]
  \draw[step=1,gridline] (0,0) grid (3,4);
  \draw[diagonal] (0,0)--(3,4);
  \draw[pathline] (0,0)--(2,0)--(2,2)--(3,2)--(3,4);
  \node[align=center,font=\scriptsize] at (1.5,-0.72)
   {$c=(1,2,5)$\\$\beta(c)=(1,4)$};
 \end{scope}

 \begin{scope}[xshift=7.35cm]
  \draw[step=1,gridline] (0,0) grid (3,4);
  \draw[diagonal] (0,0)--(3,4);
  \draw[pathline] (0,0)--(1,0)--(1,1)--(3,1)--(3,4);
  \node[align=center,font=\scriptsize] at (1.5,-0.72)
   {$c=(1,3,4)$\\$\beta(c)=(2,3)$};
 \end{scope}

 \begin{scope}[xshift=9.80cm]
  \draw[step=1,gridline] (0,0) grid (3,4);
  \draw[diagonal] (0,0)--(3,4);
  \draw[pathline] (0,0)--(1,0)--(1,1)--(2,1)--(2,2)--(3,2)--(3,4);
  \node[align=center,font=\scriptsize] at (1.5,-0.72)
   {$c=(1,3,5)$\\$\beta(c)=(2,4)$};
 \end{scope}
\end{tikzpicture}
\caption{The five below-diagonal rational $(3,4)$-Dyck paths, labeled by
their horizontal-step coordinates $c(\ell)$ and their images
$\beta(c(\ell))$.  The $\beta$-images are the vertices of
$\Omega_{3,4}$; its Hasse quiver is drawn below.}
\label{fig:34-below-paths}
\end{figure}

Figure~\ref{fig:34-below-paths} therefore fixes the labels used in the following algebraic and geometric calculations.

Let $(a,b)=(3,4)$.  Then
$s=\left\lceil\frac34\right\rceil=1, m=2$,
and $H_1=1+\left\lfloor\frac43\right\rfloor=2,
 H_2=2+\left\lfloor\frac83\right\rfloor=4$.
Thus $\Omega_{3,4}=\{12,13,14,23,24\}$,
where $ij$ denotes the pair $(i,j)$.  Its Hasse quiver is
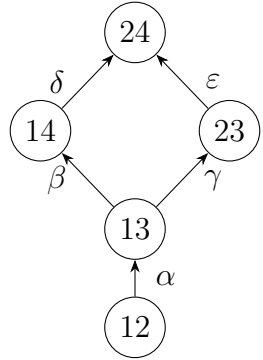
\[
\begin{tikzpicture}[
  baseline=(current bounding box.center),
  every node/.style={circle,draw,inner sep=1.5pt,minimum size=8mm},
  >=Stealth]
 \node (12) at (0,0) {$12$};
 \node (13) at (0,1.3) {$13$};
 \node (14) at (-1.25,2.6) {$14$};
 \node (23) at (1.25,2.6) {$23$};
 \node (24) at (0,3.9) {$24$};
 \draw[->] (12)--node[right,draw=none] {$\alpha$} (13);
 \draw[->] (13)--node[left,draw=none] {$\beta$} (14);
 \draw[->] (13)--node[right,draw=none] {$\gamma$} (23);
 \draw[->] (14)--node[left,draw=none] {$\delta$} (24);
 \draw[->] (23)--node[right,draw=none] {$\varepsilon$} (24);
\end{tikzpicture}
\]
The incidence category $\C_0$ has the diamond relation
\begin{equation}
 \delta\beta=\varepsilon\gamma.
\label{eq:34-diamond}
\end{equation}
The single tilting step $\C_0\simeq_{\mathrm{der}}\C_1$ introduces the
interlacing condition
\[
 x_1\leq y_1<x_2\leq y_2.
\]
Although both arrows in $12\xrightarrow{\alpha}13\xrightarrow{\gamma}23$
survive, their composite does not, because $2=y_1\nless x_2=2$
for the pair $12\to23$.  Therefore
\begin{equation}
 \C_1\cong
 \kk Q/
 \langle\gamma\alpha,\,
 \delta\beta-\varepsilon\gamma\rangle.
\label{eq:34-algebra}
\end{equation}
The relation \eqref{eq:34-diamond} then implies $\delta\beta\alpha=\varepsilon\gamma\alpha=0$,
as required by $\Hom_{\C_1}(12,24)=0$.  In this case
$\Db{\kk\Omega_{3,4}}
 \simeq
 \Db{\kk Q/
 \langle\gamma\alpha,\delta\beta-\varepsilon\gamma\rangle}$,
and the algebra on the right is the Dyck corner $B_0=eA_5^2e$.

The relative construction is also completely visible here.  Split a vertex
$ij$ into prefix $i$ and tail $j$.  The tail poset is
$\{2<3<4\}$, and
\[
 F(2)=\{1\},\qquad
 F(3)=F(4)=\{1,2\}.
\]
The tilting object is the sum of the five Kan-extended representables
\[
 (\iota_2)_*P_1,\quad
 (\iota_3)_*P_1,\quad(\iota_3)_*P_2,\quad
 (\iota_4)_*P_1,\quad(\iota_4)_*P_2.
\]

\begin{proposition}
\label{prop:34-Kan-resolutions}
Write $P_{ij}^{(0)}=\Hom_{\C_0}(ij,-)$, and denote the five summands of
the Kan-extension tilting object by $T_{ij}^{(1)}$, indexed by the
corresponding vertices of $\Omega_{3,4}$.  Then
\[
\begin{aligned}
 0&\longrightarrow P_{13}^{(0)}
   \longrightarrow P_{12}^{(0)}
   \longrightarrow T_{12}^{(1)}
   \longrightarrow0,\\
 0&\longrightarrow P_{14}^{(0)}
   \longrightarrow P_{12}^{(0)}
   \longrightarrow T_{13}^{(1)}
   \longrightarrow0,\\
 0&\longrightarrow P_{24}^{(0)}
   \longrightarrow P_{23}^{(0)}
   \longrightarrow T_{23}^{(1)}
   \longrightarrow0
\end{aligned}
\]
are minimal projective resolutions, while
$T_{14}^{(1)}\cong P_{12}^{(0)},
 T_{24}^{(1)}\cong P_{23}^{(0)}$.
\end{proposition}

\begin{proof}
Here modules are covariant $\kk$-linear functors.  Thus, for
$ij,uv\in\Omega_{3,4}$,
\[
P_{ij}^{(0)}(uv)
=
\Hom_{\C_0}(ij,uv)
=
\begin{cases}
\kk,&ij\leq uv,\\
0,&\text{otherwise}.
\end{cases}
\]
For every morphism $uv\to u'v'$, the corresponding structure map is
given by postcomposition.  Hence every nonzero structure map of
$P_{ij}^{(0)}$ is the identity after choosing the standard incidence
basis.  By Yoneda, $P_{ij}^{(0)}$ is the indecomposable projective
$\C_0$-module associated with the vertex $ij$.  The superscript $(0)$
records the category $\C_0$ and is not a cohomological degree.

For $v=(i,t)$, formula \eqref{eq:kan-formula} gives
\[
 T_v^{(1)}(a,b)
 =
 \begin{cases}
  \Hom_{F_1(t)}(i,a),&b\leq t,\\
  0,&b\nleq t.
 \end{cases}
\]
Consequently, with the vertices ordered as $12,\quad13,\quad14,\quad23,\quad24$,
the dimension vectors of the five Kan summands are
\[
\begin{array}{c|ccccc}
 &12&13&14&23&24\\ \hline
 T_{12}^{(1)}&1&0&0&0&0\\
 T_{13}^{(1)}&1&1&0&1&0\\
 T_{14}^{(1)}&1&1&1&1&1\\
 T_{23}^{(1)}&0&0&0&1&0\\
 T_{24}^{(1)}&0&0&0&1&1.
\end{array}
\]
All nonzero structure maps in these modules are identities.

The projective $P_{12}^{(0)}$ is one-dimensional at all five vertices.
Its radical is supported on
$\{13,14,23,24\}$, with all available structure maps equal to the
identity, and is therefore $P_{13}^{(0)}$.  It follows that
\[
 0\longrightarrow P_{13}^{(0)}
 \longrightarrow P_{12}^{(0)}
 \longrightarrow T_{12}^{(1)}
 \longrightarrow0
\]
is exact.

The natural quotient $P_{12}^{(0)}\to T_{13}^{(1)}$ is the identity at
$12,13,23$ and zero at $14,24$.  Its kernel is supported on
$\{14,24\}$, where its unique nonidentity structure map is the identity.
This kernel is $P_{14}^{(0)}$, giving the second resolution.

Similarly, $P_{23}^{(0)}$ is supported on $\{23,24\}$.  Its quotient
supported only at $23$ is $T_{23}^{(1)}$, and the kernel is
$P_{24}^{(0)}$.  Finally, the dimension vectors and structure maps show
directly that $T_{14}^{(1)}=P_{12}^{(0)}$ and
 $T_{24}^{(1)}=P_{23}^{(0)}$.
Each displayed differential has image contained in the radical of its
target, so the three resolutions are minimal.
\end{proof}

\begin{corollary}
Let $\mathbb L_{ij}^{(r)}$ be the objects supplied by Theorem 
\ref{thm:categorical-fukaya-staircase} for $r=0,1$.  Then
\[
\begin{aligned}
 \mathbb L_{12}^{(1)}
 &\cong
 \operatorname{Cone}\!
 \left(
  \mathbb L_{13}^{(0)}
  \longrightarrow
  \mathbb L_{12}^{(0)}
 \right),\\
 \mathbb L_{13}^{(1)}
 &\cong
 \operatorname{Cone}\!
 \left(
  \mathbb L_{14}^{(0)}
  \longrightarrow
  \mathbb L_{12}^{(0)}
 \right),\\
 \mathbb L_{23}^{(1)}
 &\cong
 \operatorname{Cone}\!
 \left(
  \mathbb L_{24}^{(0)}
  \longrightarrow
  \mathbb L_{23}^{(0)}
 \right),
\end{aligned}
\]
and $\mathbb L_{14}^{(1)}\cong\mathbb L_{12}^{(0)},
 \mathbb L_{24}^{(1)}\cong\mathbb L_{23}^{(0)}$.
Here $N=6$, and
\[
\begin{array}{c|ccccc}
v&12&13&14&23&24\\ \hline
\sigma_6(v)&56&46&36&45&35.
\end{array}
\]
Thus the initial incidence generators have the following presentations in
terms of the standard product Lagrangians:
\[
\begin{aligned}
 \mathbb L_{12}^{(0)}
 &\cong L_{36},&
 \mathbb L_{23}^{(0)}
 &\cong L_{35},\\
 \mathbb L_{13}^{(0)}
 &\cong
 \operatorname{Cone}(L_{36}\longrightarrow L_{56})[-1],&
 \mathbb L_{14}^{(0)}
 &\cong
 \operatorname{Cone}(L_{36}\longrightarrow L_{46})[-1],\\
 \mathbb L_{24}^{(0)}
 &\cong
 \operatorname{Cone}(L_{35}\longrightarrow L_{45})[-1].
\end{aligned}
\]
\end{corollary}

\begin{proof}
Applying $\Phi_0$ to the three short exact sequences in
Proposition \ref{prop:34-Kan-resolutions} gives the three distinguished triangles in
the first display, because
$\Phi_0(T_{ij}^{(1)})\cong\mathbb L_{ij}^{(1)}$
by \eqref{eq:Kan-Fukaya-image}.  The two projective Kan summands give the
two stated identifications.

By \eqref{eq:final-product-lagrangian}, $\mathbb L_v^{(1)}\cong L_{\sigma_6(v)}$.
The first two identifications therefore give
$\mathbb L_{12}^{(0)}\cong L_{36}$ and
$\mathbb L_{23}^{(0)}\cong L_{35}$.  Rotating the remaining three
distinguished triangles gives the three shifted-cone presentations of
$\mathbb L_{13}^{(0)}$, $\mathbb L_{14}^{(0)}$, and
$\mathbb L_{24}^{(0)}$.
\end{proof}

Recall that, $\mathcal K_{3,4}
=
\operatorname{thick}_{H^0(\mathcal W_6^{(2)})}
\{L_{56},L_{46},L_{36},L_{45},L_{35}\}$.
Thus $\mathcal K_{3,4}$ is the Fukaya-theoretic model of
$D^b(\Bbbk\Omega_{3,4})$ constructed in Section \ref{sec:fukaya}.

\begin{corollary}
The incidence generators have the embedded representatives
\[
 \begin{aligned}
 \mathbb L_{12}^{(0)}&\cong L_{36},&
 \mathbb L_{23}^{(0)}&\cong L_{35},\\
 \mathbb L_{13}^{(0)}&\cong(\lambda_{3,5}\times\lambda_{0,6})[-1],\\
 \mathbb L_{14}^{(0)}&\cong(\lambda_{3,4}\times\lambda_{0,6})[-1],\\
 \mathbb L_{24}^{(0)}&\cong(\lambda_{3,4}\times\lambda_{0,5})[-1].
 \end{aligned}
\]
They split-generate $\mathcal K_{3,4}$, and their full enhanced
subcategory is quasi-equivalent to $\C_0(\Omega_{3,4})^{\op}$.
\end{corollary}

\begin{proof}
In Theorem~\ref{thm:cubical-geometric-tilting}(2), take
$N=6$, $H=(2,4)$, and $\ell=(1,2)$. Then $u=3$ and
$s_0(i)=i+1$, giving all five representatives and their shifts.
Generation and formality follow from
Theorem~\ref{thm:categorical-fukaya-staircase} at $j=0$.
\end{proof}

These representatives use the product-arc triangle. A prescribed
transverse-intersection surgery or a simultaneous handle-slide
sequence would require further geometric checks; neither is needed
for this embedded collection.

\subsection{The \texorpdfstring{$(4,5)$}{(4,5)} staircase}
\label{subsec:45-geometric-example}

Let $(a,b)=(4,5)$.  Then $s=1$, $m=3$, and
$(H_1,H_2,H_3)=(2,4,6)$.
Hence $\Omega_{4,5}=
 \{1\leq z_1<z_2<z_3:
 z_1\leq2,\ z_2\leq4,\ z_3\leq6\}$.
It has fourteen vertices:
\begin{align*}
 &123,124,125,126,\quad
 134,135,136,\quad
 145,146,\\
 &234,235,236,\quad
 245,246.
\end{align*}
There are two tilting steps:
\[
 \C_0(\Omega_{4,5})
 \simeq_{\mathrm{der}}
 \C_1(\Omega_{4,5})
 \simeq_{\mathrm{der}}
 \C_2(\Omega_{4,5}),
\]
and $\operatorname{Alg}(\C_2(\Omega_{4,5}))\cong B_0$.
Their Hom conditions are
\begin{align*}
 \C_0:\quad&
 x_i\leq y_i\quad(i=1,2,3),\\
 \C_1:\quad&
 x_i\leq y_i\quad(i=1,2,3),\qquad y_1<x_2,\\
 \C_2:\quad&
 x_i\leq y_i\quad(i=1,2,3),\qquad
 y_1<x_2,\quad y_2<x_3.
\end{align*}
For example, the arrows $123\lra 124\lra 134$
are nonzero in $\C_2$, but their composite is zero because the final pair
$123\to134$ violates
\[
 y_2=3<x_3=3.
\]
This is a higher half-square zero relation. The first relative
tilting step inserts the inequality $y_1<x_2$; the second inserts
$y_2<x_3$.

In the first degree-three staircase $\Omega_{4,5}$, the bounds are
$H=(2,4,6)$.  Taking $j=1$ and $v=(1,2,4)$ in
\eqref{eq:kan-cubical-resolution} gives the concrete square
\begin{equation*}
 0\longrightarrow P_{135}^{(0)}
 \longrightarrow P_{134}^{(0)}\oplus P_{125}^{(0)}
 \longrightarrow P_{123}^{(0)}
 \longrightarrow T_{124}^{(1)}\longrightarrow0.
\end{equation*}
The resolution includes the intersection term $P_{135}^{(0)}$,
so the degree-three step is represented by the entire coherent square.  The canonical
dg lifts in \eqref{eq:enhanced-Phi-recursion} transport
\eqref{eq:kan-cubical-resolution} to a homotopy-coherent cubical
twisted-complex presentation in $\mathcal K_{4,5}$.

The same example gives a useful obstruction to a one-cone
construction using standard product branes.  In the projective bases,
direct calculation of the two tilting transforms on $K_0$ gives
\[
 \begin{aligned}
 \relax [\mathsf F_1P_{124}^{(0)}]
   &=-[P_{123}^{(1)}]+[P_{146}^{(1)}],\\
 [\mathsf F_2P_{123}^{(1)}]&=[P_{126}^{(2)}],\\
 [\mathsf F_2P_{146}^{(1)}]
   &=-[P_{145}^{(2)}]+[P_{146}^{(2)}].
 \end{aligned}
\]
Since $\sigma_8(126)=378$, $\sigma_8(145)=458$, and
$\sigma_8(146)=358$, this implies
\begin{equation*}
 [\mathbb L_{124}^{(0)}]
   =[L_{358}]-[L_{378}]-[L_{458}].
\end{equation*}
The three classes on the right are distinct elements of the
standard-projective basis of the corner's $K_0$.  Hence $\mathbb L_{124}^{(0)}$
cannot be the cone of a single morphism between two individual
standard product branes, even after shifts.  This calculation does
not obstruct a single embedded brane of a different form.

For the final-stage path $123\to124\to134$ in
$\C_2(\Omega_{4,5})$, reflection gives
$L_{568}\to L_{578}\to L_{678}$.
The only admissible permutations from $568$ to $678$ are the
identity and the transposition $s_1$: $s_2$ would require a
strand from $8$ to $7$. Thus
$CF^0(L_{568},L_{678})=\kk e_{678,568}$ and
$CF^{-1}(L_{568},L_{678})=\kk h_{s_1}$.
After choosing orientations, the theorem gives
$\mu^1(h_{s_1})=e_{678,568}$ and
$\mu^2(g,f)=e_{678,568}$.

We first identify the morphism used to construct one embedded
degree-three incidence generator. In $\C_0(\Omega_{4,5})$ one has
$T_{146}^{(1)}=P_{123}^{(0)}$ and $T_{123}^{(1)}$ is the simple
module at $123$.  Consequently,
\[
 0\longrightarrow P_{124}^{(0)}\longrightarrow T_{146}^{(1)}
 \longrightarrow T_{123}^{(1)}\longrightarrow0.
\]
For the second step,
$T_{146}^{(2)}=P_{145}^{(1)}$ and
$T_{126}^{(2)}=P_{123}^{(1)}$.  The one-coordinate
resolution $0\to P_{146}^{(1)}\to P_{145}^{(1)}
\to T_{145}^{(2)}\to0$, together with the product-arc
triangle of \cite[Lemma~3.3.4(2)]{DJL}, gives
\[
 \mathbb L_{146}^{(1)}\cong Q[-1],\qquad
 Q=\lambda_{3,4}\times\lambda_{0,5}\times\lambda_{0,8},
 \qquad \mathbb L_{123}^{(1)}\cong L_{378}.
\]
Hence a nonzero map
$\eta\in\Hom(Q[-1],L_{378})\cong\kk$, unique up to a nonzero
scalar, satisfies
\begin{equation}
 \mathbb L_{124}^{(0)}
 \cong\operatorname{Cone}
 \bigl(Q[-1]\xrightarrow{\eta}L_{378}\bigr)[-1].
\label{eq:45-124-cone}
\end{equation}

We now realize this particular cone. The geometric input is the
surgery-at-infinity triangle of \cite[Proposition~1.12]{GPS}; its
framing requirement is checked using the clean-intersection model
of \cite[Section~5.4]{GPS}.

\begin{proposition}
\label{prop:45-embedded-124}
The object $\mathbb L_{124}^{(0)}$ has an exact graded embedded
Lagrangian representative of real dimension three in
$(\operatorname{Sym}^3(D),\Lambda_8^{(3)})$.
More precisely, there are disjoint admissible representatives
$Q_0\simeq Q$ and $P_0\simeq L_{378}$ and a short Reeb chord
$\gamma$ from $Q_0$ to $P_0$, with the clean-intersection framing,
such that the embedded boundary connected sum $\Sigma$ admits
the exact triangle
$ Q_0[-1]\xrightarrow{\gamma}P_0\longrightarrow\Sigma
 \longrightarrow Q_0$.
With these gradings, $\Sigma[-1]\cong\mathbb L_{124}^{(0)}$.
\end{proposition}

\begin{proof}
Write the source arcs as
$A_1=\lambda_{3,4}$, $A_2=\lambda_{0,5}$,
$A_3=\lambda_{0,8}$, and the target arcs as
$B_1=\lambda_{0,3}$, $B_2=\lambda_{0,7}$,
$B_3=\lambda_{0,8}$. We use the clockwise boundary-label
convention of \cite{DJL}; positive Reeb flow is counterclockwise.
The two copies of $\lambda_{0,8}$ have separate endpoints.
Within each stop-free boundary interval, choose the following
clockwise orders:
\[
\begin{array}{c|c}
 \text{interval}&\text{endpoint order}\\ \hline
 0&A_3,\ A_2,\ B_3,\ B_2,\ B_1\\
 3&B_1,\ A_1\\
 8&A_3,\ B_3.
\end{array}
\]
The remaining endpoints are $A_1$ in interval $4$, $A_2$ in
interval $5$, and $B_2$ in interval $7$. Use chord representatives
in the disk, smoothing and cylindricalizing their ends within the
same stop-free intervals. All three source arcs are mutually
disjoint, as are the three target arcs. Alternation of endpoints
gives the matrix of interior intersections
\[
 (\#(A_i\cap B_j))_{i,j}
 =\begin{pmatrix}0&0&0\\0&1&1\\0&0&1\end{pmatrix}.
\]
There is no bijective matching, so their products $Q_0$ and
$P_0$ are disjoint in the symmetric product.

Let $p_2$ be the unique intersection of $A_2$ with $B_2$,
and $p_3$ that of $A_3$ with $B_3$. These points are distinct
because the source arcs are disjoint. At the face where the first
coordinate approaches boundary interval $3$, the other two
coordinates have exactly this one transverse matching.
This face is off the diagonal. Near it, choose a boundary
coordinate $z$ in interval $3$ and a local contact form
$dz+\beta$, where $\beta$ is the sum of the two surface
Liouville forms, with the boundary radius absorbed into the
coordinates. If $df_Q$ and $df_P$ are the restrictions of
$\beta$ to the two product sheets, their cylindricalizations
have Legendrian lifts $z=c_Q-f_Q$ and $z=c_P-f_P$.
Normalize both local primitives to vanish at $(p_2,p_3)$.
Choosing $c_Q=c_P$ brings the two interval-$3$ endpoints to
the same contact height. Their projections meet transversely
at exactly $(p_2,p_3)$, so the lifts have an isolated clean
intersection. The other endpoint separations are unchanged,
and the boundary Legendrians are otherwise disjoint.
In contact Darboux coordinates
$\alpha=dz-\xi_2\,dq_2-\xi_3\,dq_3$, these lifts can be
straightened to
\[
 \overline\Lambda_Q=\{z=0,\ \xi_2=\xi_3=0\},\qquad
 \overline\Lambda_P=\{z=0,\ q_2=q_3=0\}.
\]
Take small negative and positive Reeb pushoffs of these lifts,
respectively, and extend them by admissible isotopies of the
original cylindrical ends. Their central chord is $\gamma$.
Thus the chord and its untwisted product framing satisfy the
clean-intersection construction of \cite[Section~5.4]{GPS}.
The chart lies entirely over interval $3$ and away from the stop.

To identify its morphism, use the standard product-arc Floer
model, with the continuation identification for the small
cylindricalization above. Positively push the source endpoint
past $B_1$ in interval $3$, keeping the other endpoint orders
fixed. The intersection matrix becomes
\[
 \begin{pmatrix}1&0&0\\0&1&1\\0&0&1\end{pmatrix}.
\]
The only bijective matching is
$A_1\mapsto B_1$, $A_2\mapsto B_2$, $A_3\mapsto B_3$.
Its three intersection points are distinct because the source
arcs are disjoint. It gives one transverse intersection of the
product branes.
Continuing the source endpoints toward the next stops preserves
these orders and this unique matching; this is the usual
cofinal positive-push model for the partially wrapped morphism.
Consequently its Floer complex has one generator $x$ and zero
differential. The continuation maps between successive such
pushes preserve $x$: subdivide the isotopy into small steps,
each with its unique local continuation strip and no births
or deaths. Under the short-chord/intersection identification,
$\gamma$ becomes $x$. In the product gradings fixed above, its
interval-$3$ factor is the degree-one connecting class of
the arc triangle
\[
 \lambda_{0,3}\longrightarrow\lambda_{0,4}
 \longrightarrow\lambda_{3,4}\longrightarrow\lambda_{0,3}[1].
\]
The other two factors, from $\lambda_{0,5}$ to
$\lambda_{0,7}$ and from $\lambda_{0,8}$ to itself, have degree
zero. Thus $|x|=1$ before the source shift and zero after it.
The single-generator complex and the already established
$\Hom(Q[-1],L_{378})\cong\kk$ identify its class as a nonzero
scalar multiple of $\eta$. Thus both the morphism and its
grading agree with the specified categorical cone.

Attach the relatively non-exact embedded $1$-handle of
\cite[Section~5.4]{GPS}. Since it joins two distinct connected
components, their primitive constants can be chosen independently
to glue the handle primitive, giving a globally exact Lagrangian
$\Sigma$. The source shift above gives the compatible brane
grading; the Pin structures extend across the boundary connected
sum. The construction includes the small perturbation outside
the Darboux chart required in \cite[Remark~5.3]{GPS}. Choose it
smaller than the positive clearance of both cylindrical ends
from the closed stop. Thus the entire attachment remains
admissible, including that perturbation. Embeddedness is part
of this handle construction.

The surgery exact triangle
\cite[Proposition~1.12]{GPS} has first map precisely the central
chord $\gamma$. It therefore gives the stated triangle and
$\Sigma\cong\operatorname{Cone}(\eta)$, since rescaling a nonzero
map by a scalar does not change its cone. Equation
\eqref{eq:45-124-cone} proves the asserted identification.
\end{proof}

\begin{remark}
This single-object application uses the $A_\infty$-categorical
surgery triangle to identify the Fukaya object. A simultaneous
realization of the incidence collection still requires compatible
geometric moves.
\end{remark}

\subsection{Cartan and multiplication checks}

We conclude the examples with finite checks of the derived
invariants and multiplication. Order the vertices lexicographically.
As in Proposition~\ref{prop:noninterval-counterexample}, use targets
for rows and sources for columns in the Cartan matrix:
\begin{equation}
 (C_j)_{y,x}
 =
 \dim_{\kk}\Hom_{\C_j(\Omega)}(x,y)
 =
 \begin{cases}
  1,&\begin{gathered}x_i\leq y_i\ (1\leq i\leq m),\\[-2pt]
      y_i<x_{i+1}\ (1\leq i\leq j),\end{gathered}\\
  0,&\text{otherwise}.
 \end{cases}
\label{eq:cartan-enumeration}
\end{equation}

These zero--one matrices are lower unitriangular and have determinant
one. For the inverse Coxeter transformation, use
$\Phi_j=-C_j^{-\mathsf T}C_j$, direct integer matrix calculation gives the table below. The
three dimensions in the $(4,5)$ row correspond to
$\C_0,\C_1,\C_2$, respectively. The convention is uniform with
the earlier counterexample.
\begin{equation*}
\resizebox{0.98\textwidth}{!}{$
\begin{array}{c|c|c|c|c}
 (a,b)&|\Omega_{a,b}|&
 \dim_{\kk}\C_0,\ldots,\dim_{\kk}\C_{m-1}&
 \det C_0,\ldots,\det C_{m-1}&
 \det(tI-\Phi_j)\\ \hline
 (3,4)&5&14,\ 12&1,\ 1&(t+1)(t^4+1)\\
 (4,5)&14&84,\ 68,\ 55&1,\ 1,\ 1&p_{4,5}(t)
\end{array}$}
\end{equation*}
Here the common Coxeter polynomial in the second row is
\begin{equation*}
\begin{split}
 p_{4,5}(t)={}&t^{14}+t^{13}+t^{12}+t^{11}+t^{10}
 +2t^9+2t^8+2t^7\\
 &+2t^6+2t^5+t^4+t^3+t^2+t+1.
\end{split}
\end{equation*}
Thus the Cartan determinants and Coxeter polynomials agree at every tilting
step, as required by derived equivalence.

The enumeration also checks every product of basis morphisms.
 For every ordered triple $(x,y,z)$, retain it when
$\Hom(x,y)$ and $\Hom(y,z)$ are nonzero, and declare the product to be
the basis element of $\Hom(x,z)$ exactly when the inequalities in
\eqref{eq:cartan-enumeration} also hold for $(x,z)$.  In the final categories, counting identities as well, there are
$25$ composable pairs for $(3,4)$, of which $3$ have zero product;
for $(4,5)$ there are $188$, of which $48$ have zero product.  The resulting multiplication tables have dimensions $12$ and
$55$, respectively, and agree entry by entry with the corner
multiplication rule \eqref{eq:higher-composition}.  In particular, the
relations $\gamma\alpha=0$ in \eqref{eq:34-algebra} and
$123\to124\to134=0$ are instances of the complete, rather than a sampled,
check.


\begin{thebibliography}{99}

\bibitem{CLR}
F.~Chapoton, S.~Ladkani, and B.~Rognerud,
\emph{On derived equivalences for categories of generalized intervals of a finite poset},
preprint, \href{https://arxiv.org/abs/1801.05154}{arXiv:1801.05154} (2018).


\bibitem{DiDedda}
I.~Di Dedda,
\emph{Symplectic higher Auslander correspondence for type $A$},
Quantum Topol. (2026), published online first,
\href{https://doi.org/10.4171/QT/253}{doi:10.4171/QT/253}.
Result numbers refer to \href{https://arxiv.org/abs/2311.16859v3}{arXiv:2311.16859v3}.


\bibitem{DJL}
T.~Dyckerhoff, G.~Jasso, and Y.~Lekili,
\emph{The symplectic geometry of higher Auslander algebras: symmetric
products of disks},
Forum Math. Sigma \textbf{9} (2021), Paper No.~e10, 49~pp.
\href{https://doi.org/10.1017/fms.2021.2}{doi:10.1017/fms.2021.2}.


\bibitem{GPS}
S.~Ganatra, J.~Pardon, and V.~Shende,
\emph{Sectorial descent for wrapped Fukaya categories},
J. Amer. Math. Soc. \textbf{37} (2024), no.~2, 499--635.
\href{https://doi.org/10.1090/jams/1035}{doi:10.1090/jams/1035}.
For the surgery statement and its local model we use the authors'
\href{https://web.math.princeton.edu/~jpardon/manuscripts/14_descent.pdf}
{manuscript}, Proposition~1.12 and Section~5.4.

\bibitem{Goguet}
M.~Goguet,
\emph{Derived equivalence of posets of torsion classes},
\href{https://arxiv.org/abs/2606.21239v1}{arXiv:2606.21239v1} (2026).


\bibitem{Gottesman}
T.~Gottesman,
\emph{Fractionally Calabi--Yau lattices that tilt to higher Auslander
algebras of type $A$},
Adv. Math. \textbf{488} (2026), Article~110785,
\href{https://doi.org/10.1016/j.aim.2026.110785}{doi:10.1016/j.aim.2026.110785}.


\bibitem{Iyama2007}
O.~Iyama,
\emph{Auslander correspondence},
Adv. Math. \textbf{210} (2007), no.~1, 51--82.
\href{https://doi.org/10.1016/j.aim.2006.06.003}{doi:10.1016/j.aim.2006.06.003}.


\bibitem{Iyama2011}
O.~Iyama,
\emph{Cluster tilting for higher Auslander algebras},
Adv. Math. \textbf{226} (2011), no.~1, 1--61.
\href{https://doi.org/10.1016/j.aim.2010.03.004}{doi:10.1016/j.aim.2010.03.004}.


\bibitem{IyamaOppermann}
O.~Iyama and S.~Oppermann,
\emph{$n$-representation-finite algebras and $n$-APR tilting},
Trans. Amer. Math. Soc. \textbf{363} (2011), no.~12, 6575--6614.
\href{https://arxiv.org/abs/0909.0593v2}{arXiv:0909.0593v2}.


\bibitem{JassoKulshammer}
G.~Jasso and J.~K\"ulshammer,
\emph{Higher Nakayama algebras I: Construction},
with appendices by J.~K\"ulshammer and C.~Psaroudakis, and by S.~Kvamme,
Adv. Math. \textbf{351} (2019), 1139--1200.
\href{https://doi.org/10.1016/j.aim.2019.05.026}
{doi:10.1016/j.aim.2019.05.026}.


\bibitem{KellerAInfinity}
B.~Keller,
\emph{Introduction to $A_\infty$ algebras and modules},
Homology Homotopy Appl. \textbf{3} (2001), no.~1, 1--35,
\href{https://arxiv.org/abs/math/9910179}{arXiv:math/9910179}.

\bibitem{KellerDG}
B.~Keller,
\emph{On differential graded categories},
in \emph{International Congress of Mathematicians}, Vol.~II,
European Mathematical Society, Z\"urich, 2006, 151--190,
\href{https://arxiv.org/abs/math/0601185}{arXiv:math/0601185}.

\bibitem{Ladkani}
S.~Ladkani,
\emph{On derived equivalences of lines, rectangles and triangles},
J. Lond. Math. Soc. (2) \textbf{87} (2013), no.~1, 157--176.
\href{https://doi.org/10.1112/jlms/jds034}{doi:10.1112/jlms/jds034}.


\bibitem{MinamotoYamaura}
H.~Minamoto and K.~Yamaura,
\emph{Homological dimension formulas for trivial extension algebras},
J. Pure Appl. Algebra \textbf{224} (2020), no.~8, Article~106344, 30~pp.,
\href{https://doi.org/10.1016/j.jpaa.2020.106344}{doi:10.1016/j.jpaa.2020.106344}.


\bibitem{OppermannThomas}
S.~Oppermann and H.~Thomas,
\emph{Higher-dimensional cluster combinatorics and representation theory},
J. Eur. Math. Soc. \textbf{14} (2012), no.~6, 1679--1737.
\href{https://doi.org/10.4171/JEMS/345}{doi:10.4171/JEMS/345}.


\bibitem{Rickard1989}
J.~Rickard,
\emph{Morita theory for derived categories},
J. Lond. Math. Soc. (2) \textbf{39} (1989), no.~3, 436--456.
\href{https://doi.org/10.1112/jlms/s2-39.3.436}{doi:10.1112/jlms/s2-39.3.436}.


\bibitem{Rickard1991}
J.~Rickard,
\emph{Derived equivalences as derived functors},
J. Lond. Math. Soc. (2) \textbf{43} (1991), no.~1, 37--48.
\href{https://doi.org/10.1112/jlms/s2-43.1.37}{doi:10.1112/jlms/s2-43.1.37}.


\bibitem{Xing}
W.~Xing,
\emph{Replicated algebras derived equivalent to higher Auslander algebras of type $A$},
preprint, \href{https://arxiv.org/abs/2511.22655v1}{arXiv:2511.22655v1} (2025).


\end{thebibliography}
\end{document}